\documentclass{daj}
\dajAUTHORdetails{%
  title = {Stability of Ranks under Field Extensions}, 
  author = {Qiyuan Chen and Ke Ye},
  plaintextauthor = {Qiyuan Chen,  Ke Ye},
  plaintexttitle = {Stability of ranks under field extensions}, 
  runningtitle = {Stability of ranks under field extensions}, 
  runningauthor = {Qiyuan Chen and Ke Ye},
  copyrightauthor = {Qiyuan Chen and Ke Ye},
  keywords = {analytic rank,  partition rank,  slice rank,  geometric rank,  tensor,  multilinear function,  stability,  field extension},
}   

\dajEDITORdetails{%
   year={2025},
   number={27},
   received={20 March 2025},   
   published={17 December 2025},  
   doi={10.19086/da.145807},       
}   
\usepackage{amsthm,amssymb,amsfonts}
\usepackage[nosumlimits]{mathtools}
\usepackage{mathrsfs}
\usepackage{enumerate}
\usepackage{adjustbox}
\usepackage{tikz-cd}

\theoremstyle{plain}
\newtheorem{theorem}{Theorem}[section]
\newtheorem{proposition}[theorem]{Proposition}
\newtheorem{lemma}[theorem]{Lemma}
\newtheorem{corollary}[theorem]{Corollary}
\newtheorem{remark}[theorem]{Remark}
\theoremstyle{definition}

\newtheorem{conjecture}[theorem]{Conjecture}

\DeclareMathOperator{\sr}{SR}
\DeclareMathOperator{\pr}{PR}
\DeclareMathOperator{\Div}{Div}
\DeclareMathOperator{\Ker}{Ker}
\DeclareMathOperator{\id}{id}
\DeclareMathOperator{\AR}{AR}
\DeclareMathOperator{\GR}{GR}
\DeclareMathOperator{\Id}{Id}
\DeclareMathOperator{\R}{R}
\DeclareMathOperator{\Q}{Q}
\DeclareMathOperator{\codim}{codim}
\DeclareMathOperator{\spa}{span}
\DeclareMathOperator{\Hom}{Hom}
\DeclareMathOperator{\tr}{tr}

\begin{document}
\begin{frontmatter}[classification=text]

\title{Stability of Ranks under Field Extensions} 

\author[QYC]{Qiyuan Chen\thanks{Supported by the National Key Research Project of China under Grant No.  2020YFA0712300 and the National Natural Science Foundation of China No. 12288201}}
\author[KY]{Ke Ye\thanks{Supported by the National Key Research Project of China under Grant No.  2020YFA0712300 and the National Natural Science Foundation of China No. 12288201}}

\begin{abstract}
This paper studies the stability of tensor ranks under field extensions.  Our main contributions are fourfold: (1) We prove that the analytic rank is stable under field extensions.  (2) We establish the equivalence between the partition rank vs.  analytic rank conjecture and the stability conjecture for partition rank.  We also prove that they are equivalent to other two important conjectures.  (3) We resolve the Adiprasito-Kazhdan-Ziegler conjecture on the stability of the slice rank of linear subspaces under field extensions.  (4) As an application of (1),  we show that the geometric rank is equal to the analytic rank up to a constant factor.
\end{abstract}
\end{frontmatter}

\section{Introduction}
Given a tensor $T \in \mathbb{K}^{n_1} \otimes_{\mathbb{K}} \cdots \otimes_{\mathbb{K}}  \mathbb{K}^{n_d}$,  we denote by $\pr_{\mathbb{K}}(T)$ (resp.  $\sr_{\mathbb{K}}(T)$,  $\AR_{\mathbb{K}}(T)$,  $\GR_{\mathbb{K}}(T)$) the \emph{partition rank \cite{Naslund20} (resp.  slice rank \cite{Terrence16},  analytic rank \cite{GW11},  geometric rank \cite{KMZ23})} of $T$ (cf. ~\eqref{eq:R+Q}--\eqref{eq:AR}).  These ranks were first introduced to study various problems in additive combinatorics \cite{GT08,GT09, GW11,Terrence16,  CL17, EG17,  Naslund20}, but very soon they also received great attention from other fields such as algebraic geometry \cite{KZ21,GL22},  coding theory \cite{KL08,BL23} and algebraic complexity theory \cite{BCCGNSU17,KMZ23}. It is worth noting that the partition rank is equivalent to the \emph{Schmidt rank},  as tensors can be viewed as multilinear polynomials.  The notion of Schmidt rank was first introduced to explore geometric and arithmetic properties of homogeneous polynomials in \cite{Schmidt85}.  Later,  it was extended to inhomogeneous polynomials \cite{KZ18,adiprasito2021schmidt}.  The \emph{strength} of polynomials is an analogue of the partition rank,  proposed in \cite{AH20} to study invariants of polynomial ideals.  Both strength and Schmidt rank have played central roles in understanding the behaviour of polynomials \cite{GT09,HS10, TZ12,Milicevic19, Draisma19,AH20,BBOV23,KLP23, KP23}.

Arguably,  the most fundamental problems concerning these ranks are  their relations to each other and their stability under field extensions.  The \emph{stability} of a function $f$ of tensors under field extensions means the existence of two constants $c \coloneqq c(d,  \mathbb{K}) > 0$ and $C \coloneqq C(d,  \mathbb{K}) > 0$ such that for any $T\in \mathbb{K}^{n_1} \otimes_{\mathbb{K}}  \cdots \otimes_{\mathbb{K}}  \mathbb{K}^{n_d}$ and any (finite) algebraic extension $\mathbb{F}/\mathbb{K}$,  we have 
\[
c f(T) \le f \left( T^{\mathbb{F}} \right) \le C f(T),
\]
where $T^{\mathbb{F}}$ is the tensor in $\mathbb{F}^{n_1} \otimes_{\mathbb{F}}  \cdots \otimes_{\mathbb{F}} \mathbb{F}^{n_d}$ determined by $T$.  Here and in the rest of the paper,  we use $C(d,\mathbb{K})$ to indicate the dependence of the constant $C$ on $d$ and $\mathbb{K}$.   Furthermore,  when two tensor functions $f$ and $g$ satisfy $cf \le g \le C f$ for some positive constants $c$ and $C$,  we denote $f  \asymp g$.

We remark that in the context of algebraic geometry,  the stability of $f$ may be rephrased as the invariance of the order of the magnitude of $f$ under base changes.  

Although the stability of ranks plays an essential role in existing works \cite{BCS97,  CGLM08,  CLQY20,  adiprasito2021schmidt,moshkovitz2022quasi,derksen2022g,cohen2023partition}, there is no systematic investigation of the stability in the literature,  to the best of our knowledge.  The goal of this paper is to study the stability of aforementioned tensor ranks,  and establish relations among them by the stability.  To better illustrate our contributions,  we first recall the following three well-known conjectures. 
\begin{conjecture}[Partition rank vs.  analytic rank conjecture]\cite{GW11,Milicevic19, adiprasito2021schmidt}    \label{conj1-1}
There exists a constant $C \coloneqq C(d,q) > 0$ such that 
\[
\pr_{\mathbb{F}_q}(T)\le C \AR_{\mathbb{F}_{q}}(T)
\] 
for any $T \in \mathbb{F}_q^{n_1} \otimes_{\mathbb{F}_q} \cdots \otimes_{\mathbb{F}_q} \mathbb{F}_q^{n_d}$.
\end{conjecture}
\begin{conjecture}[Partition rank vs.  geometric rank conjecture]
\cite{Schmidt85, adiprasito2021schmidt, cohen2023partition}\label{conj1-3}
There exists a constant $C \coloneqq C (d,q) >0$ such that 
\[
\pr_{\mathbb{F}_{q}}(T)\le C \GR_{\mathbb{F}_{q}}(T)
\]        
for any $T \in \mathbb{F}_q^{n_1} \otimes_{\mathbb{F}_q} \cdots \otimes_{\mathbb{F}_q} \mathbb{F}_q^{n_d}$. 
\end{conjecture}
\begin{conjecture}[Stability conjecture for partition rank]\cite{adiprasito2021schmidt}\label{conj1-2} There exists a constant $C \coloneqq C(d,q) > 0$ such that 
\[
\pr_{\mathbb{F}_{q}}(T)\le C  \pr_{\overline{\mathbb{F}}_{q}} \left( T^{\overline{\mathbb{F}}_{q}} \right)
\]
for any $T \in \mathbb{F}_q^{n_1} \otimes_{\mathbb{F}_q} \cdots \otimes_{\mathbb{F}_q} \mathbb{F}_q^{n_d}$.
\end{conjecture}
We remark that the original formulations of Conjectures~\ref{conj1-1}--\ref{conj1-2} actually require the constant $C$ to be only dependent on $d$.  However,  for instance,  it was proved in \cite{cohen2023partition} that  the original version of Conjecture~\ref{conj1-1} is true over finite fields of sufficiently large cardinality depending on the analytic rank.  Thus,  to completely resolve the original version of Conjecture~\ref{conj1-1},  one needs to remove the dependence of the result in \cite{cohen2023partition} on the analytic rank and prove Conjecture~\ref{conj1-1} for fields of small cardinalities.  We also want to mention that according to \cite[Remark~4]{Terrence09},  algebraic problems over fields of small cardinalities are usually considered more difficult than those over sufficiently large fields. 
\subsection*{Main results}
According to \cite[Theorem~1.13]{adiprasito2021schmidt},  Conjecture~\ref{conj1-2} implies Conjecture~\ref{conj1-1},  assuming that Proposition~$\text{III}_{\mathbb{C}}$ in \cite{Schmidt85} is valid for arbitrary algebraically closed fields of finite characteristic.  Only very recently,  it was realized that this assumption is not necessary \cite{cohen2023partition,KLP23}.  In Theorem~\ref{thm4-6},  we will prove  that Conjectures~\ref{conj1-1}--\ref{conj1-2} are all equivalent.  In fact,  Theorem~\ref{thm4-6} proves that Conjectures~\ref{conj1-1}--\ref{conj1-2} are all equivalent to 
\begin{conjecture}[Asymptotic direct sum conjecture for partition rank]\label{conj1-4}
There exists a constant $C \coloneqq C(d,q) > 0$ such that 
\[
\pr_{\mathbb{F}_q} (T) \le C \limsup_{n\to\infty} \frac{\pr_{\mathbb{F}_{q}}\left( T^{\oplus n} \right)}{n} 
\]        
for any $T \in \mathbb{F}_q^{n_1} \otimes_{\mathbb{F}_q}  \cdots \otimes_{\mathbb{F}_q}  \mathbb{F}_q^{n_d}$. 
\end{conjecture}
It is noticeable that Conjecture~\ref{conj1-4} is a weaker version of the direct sum conjecture for partition rank \cite{gowers2021slice},  which is an analogue of the direct sum conjecture for cp-rank \cite{Strassen73,  JT86,  LM17},  disproved by \cite{Shitov19}.  In particular, Conjecture~\ref{conj1-4} is true for slice rank \cite{gowers2021slice}.  It is worth noting that partition rank coincides with slice rank for $d=3$.  Consequently,  Theorem~\ref{thm4-6} directly implies the validity of Conjectures~\ref{conj1-1}--\ref{conj1-2} for $d = 3$,  results that had been established in \cite{CM21,  adiprasito2021schmidt,derksen2022g}. 

One of the main ingredients in the proof of Theorem~\ref{thm4-6} is the stability of the analytic rank,  proved in Theorem~\ref{thm3-5}.  In earlier works such as \cite[Proposition~8.2]{moshkovitz2022quasi} and \cite[Proof of Corollary~1]{cohen2023partition},  it was proved that for any tensor $T \in \mathbb{F}_q^{n_1} \otimes_{\mathbb{F}_q} \cdots\otimes_{\mathbb{F}_q} \mathbb{F}_q^{n_d}$ and integers $l \ge 1,  d \ge 2$,  
\[
\AR_{\mathbb{F}_{q^{l}}}\left( T^{\mathbb{F}_{q^{l}}}\right) \le l^{d-2} \AR_{\mathbb{F}_{q}}(T).  
\]
Theorem~\ref{thm3-5} not only improves the coefficient $l^{d-1}$ to a constant depending only on $d$ and $q$,  but it also shows that $\AR_{\mathbb{F}_{q^{l}}}\left( T^{\mathbb{F}_{q^{l}}}\right)$ is bounded below by $\AR_{\mathbb{F}_{q}}(T)$,  up to a constant independent to $l$.  In Theorem~\ref{thm3-7},  we show  that $\lim_{ l \to\infty} \AR_{\mathbb{F}_{q^l}}\left( T^{\mathbb{F}_{q^l}} \right)= \GR_{\mathbb{F}_q}(T)$ ,  which is an analogue of \cite[Theorem~8.1]{KMZ23}. As an interesting application of Theorems~\ref{thm3-5} and \ref{thm3-7},  we prove in Proposition~\ref{coro3-8} that $\AR_{\mathbb{F}_{q}}(T) \asymp \GR_{\mathbb{F}_{q}}(T)$.  Previously,  it was only known that $\GR_{\mathbb{F}_{q}}(T)$ is bounded above by a constant multiple of $\AR_{\mathbb{F}_{q}}(T)$,   see \cite[Theorem~1]{CM21} and \cite[Proof of Theorem~1.13~(1)]{adiprasito2021schmidt}. Meanwhile,  \cite[Theorem~1.2]{baily2024strength} independently establishes the linear relation between geometric rank and analytic rank,  using tools from number theory and algebraic geometry.

As a corollary of the proof of Theorem~\ref{thm4-6},  we obtain Proposition~\ref{coro4-4} concerning the stability of slice rank.  Although a stronger stability result \cite{derksen2022g} may be obtained  by geometric invariant theory,  our proof of Proposition~\ref{coro4-4} is much more elementary.  The slice rank $\sr_{\mathbb{K}}(\mathsf{W})$ (cf. ~\eqref{eq:SRW}) of a linear subspace $\mathsf{W} \subseteq \mathbb{K}^{n_1} \otimes_{\mathbb{K}} \cdots \otimes_{\mathbb{K}}  \mathbb{K}^{n_d}$ was introduced in  \cite[Definition~1.3]{adiprasito2021schmidt} and it was conjectured that $\sr_{\mathbb{K}}(\mathsf{W})$ is also stable under field extensions.  In particular,  if $\mathsf{W} = \spa\{ T \}$,  then we have $\sr_{\mathbb{K}}(\mathsf{W}) = \sr_{\mathbb{K}} (T)$ by \cite[Claim~1.4]{adiprasito2021schmidt}.
\begin{conjecture}\cite[Conjecture~1.16]{adiprasito2021schmidt}\label{conj1-5}
There is a constant $C \coloneqq C (d)>0$ such that for any field $\mathbb{K}$ and linear subspace $\mathsf{W} \subseteq \mathbb{K}^{n_1} \otimes_{\mathbb{K}} \cdots \otimes_{\mathbb{K}} \mathbb{K}^{n_d}$,  we have 
\[
\sr_{\mathbb{K}}(\mathsf{W})\le C\sr_{\overline{\mathbb{K}}}(\mathsf{W} \otimes_{\mathbb{K}} \overline{\mathbb{K}}).
\]
\end{conjecture}
Conjecture~\ref{conj1-5} was only confirmed for $d = 2$ in \cite[Theorem~1.17]{adiprasito2021schmidt}.  In Theorem~\ref{thm:slrank-conj},  we resolve Conjecture~\ref{conj1-5} with $C = d/2 + 1$ for all $d \ge 2$,  by associating a tensor $T_{\mathsf{W}}$ of order $d+2$ to $\mathsf{W}$ such that $\sr_{\mathbb{K}}(T_\mathsf{W}) = \sr_{\mathbb{K}}(\mathsf{W})$. 

To conclude this section,  we summarize the main results of this paper in Figure~\ref{fig:summary},  where $T$ is a tensor and $\mathsf{W}$ is a linear subspace of tensors,  $\mathbb{F}$ is a finite extension of $\mathbb{F}_q$.  The symbol ``$\asymp$" means equal up to a constant factor,  ``$\asymp?$" means ``$\asymp$" is conjectured to hold and all conjectures around $\pr_{\mathbb{F}_q}(T)$ are equivalent.  
\begin{figure}[!htbp] 
\adjustbox{scale=1.1,center}{
\begin{tikzcd} 
	{\operatorname{SR}_{\overline{\mathbb{K}}}\left( \mathsf{W} \otimes_{\mathbb{K}} \overline{\mathbb{K}} \right)} && {\operatorname{SR}_{\mathbb{K}}(\mathsf{W})} \\
	\\
	{\operatorname{SR}_{k,\overline{\mathbb{K}}}\left( \mathsf{W} \otimes_{\mathbb{K}} \overline{\mathbb{K}} \right)} && {\operatorname{SR}_{k,\mathbb{K}}(\mathsf{W})} \\
	\\
	{\operatorname{SR}_{\overline{\mathbb{K}}}\left(T^{\overline{\mathbb{K}}}\right)} && {\operatorname{SR}_{\mathbb{K}}(T)} && {\operatorname{PR}_{\mathbb{F}_q}(T)} && {\operatorname{PR}_{\overline{\mathbb{F}}_q}\left(T^{\overline{\mathbb{F}}_q}\right)} \\
	\\
	{\operatorname{AR}_{\mathbb{F}}\left(T^{\mathbb{F}}\right)} && {\operatorname{AR}_{\mathbb{F}_q}(T)} && {\operatorname{GR}_{\mathbb{F}_q}(T)} && {\operatorname{GR}_{\overline{\mathbb{F}}_q}\left(T^{\overline{\mathbb{F}}_q}\right)}
	\arrow["{{\asymp}}"{marking, allow upside down}, shift right=2, draw=none, from=1-1, to=1-3]
	\arrow["{\text{\scriptsize{Thm.~\ref{thm:slrank-conj}}}}"{marking, allow upside down}, shift left=2, draw=none, from=1-1, to=1-3]
	\arrow["\ge"{marking, allow upside down}, shift right=5, draw=none, from=1-3, to=3-3]
	\arrow["{\text{\eqref{eq:SRW}}}"{description}, shift left=5, draw=none, from=1-3, to=3-3]
	\arrow["{\asymp?}"{marking, allow upside down}, shift right=2, draw=none, from=3-1, to=3-3]
	\arrow["{\text{\scriptsize{Conj.~\ref{conj6}}}}"{marking, allow upside down}, shift left=2, draw=none, from=3-1, to=3-3]
	\arrow["\ge"{marking, allow upside down}, shift right=5, draw=none, from=3-3, to=5-3]
	\arrow["{\text{\eqref{eq:SRkW}}}"{description}, shift left=5, draw=none, from=3-3, to=5-3]
	\arrow["\asymp"{marking, allow upside down}, shift right=2, draw=none, from=5-1, to=5-3]
	\arrow["{\text{\scriptsize{Prop.~\ref{coro4-4}}}}"{marking, allow upside down}, shift left=2, draw=none, from=5-1, to=5-3]
	\arrow["{{\ge }}"{marking, allow upside down}, shift right=2, draw=none, from=5-3, to=5-5]
	\arrow["{\text{ \scriptsize{$\mathbb{K} = \mathbb{F}_q$ }}}"{marking, allow upside down}, shift left=2, draw=none, from=5-3, to=5-5]
	\arrow["{{\asymp?}}"{marking, allow upside down}, shift right=2, draw=none, from=5-5, to=5-7]
	\arrow["{\text{\scriptsize{Conj.~\ref{conj1-2}}}}"{marking, allow upside down}, shift left=2, draw=none, from=5-5, to=5-7]
	\arrow["{\asymp?}"{marking, allow upside down}, shift right=5, draw=none, from=5-5, to=7-5]
	\arrow["{\text{\scriptsize{Conj.~\ref{conj1-3}}}}"{description}, shift left=5, draw=none, from=5-5, to=7-5]
	\arrow["\asymp"{marking, allow upside down}, shift right, draw=none, from=5-7, to=7-7]
	\arrow["{\text{\scriptsize{\cite{cohen2023partition}}}}"{description}, shift left = 3.5, draw=none, from=5-7, to=7-7]
	\arrow["{\text{\scriptsize{Thm.~\ref{thm3-5}}}}"{marking, allow upside down}, shift left=2, draw=none, from=7-1, to=7-3]
	\arrow["\asymp"{marking, allow upside down}, shift right=2, draw=none, from=7-1, to=7-3]
	\arrow["{{\asymp ?}}"{marking, allow upside down}, shift right=2, draw=none, from=7-3, to=5-5]
	\arrow["{\text{\scriptsize{Conj.~\ref{conj1-1}}}}"{marking, allow upside down}, shift left=2, draw=none, from=7-3, to=5-5]
	\arrow["\asymp"{marking, allow upside down}, shift right=2, draw=none, from=7-3, to=7-5]
	\arrow["\begin{array}{c} \substack{\text{\scriptsize{Thm.~\ref{thm3-7}}} \\ \text{\scriptsize{Prop.~\ref{coro3-8}} }} \end{array}"{marking, allow upside down}, shift left=2, draw=none, from=7-3, to=7-5]
	\arrow["{ = }"{marking, allow upside down}, shift right=2, draw=none, from=7-5, to=7-7]
	\arrow["{\text{\scriptsize{\eqref{eq:GR}}}}"{marking, allow upside down}, shift left=2, draw=none, from=7-5, to=7-7]
\end{tikzcd}
}
\caption{Summary of results}
\label{fig:summary}
\end{figure}

\section{Notations}\label{sec:prelim}
We first recall three basic operations. 
Given a permutation $\pi \in \mathfrak{S}_d$,  an integer $1 \le p \le n$ and tensors $T = u_1 \otimes \cdots \otimes u_d \in \mathbb{K}^{n_1} \otimes_{\mathbb{K}} \cdots \otimes_{\mathbb{K}} \mathbb{K}^{n_d}$,  $S = v_1 \otimes \cdots \otimes v_d \in \mathbb{K}^{m_1} \otimes_{\mathbb{K}} \cdots \otimes_{\mathbb{K}} \mathbb{K}^{m_d}$,  $F = f_{1} \otimes \cdots \otimes f_p  \in (\mathbb{K}^{n_{\pi(1)}})^\ast \otimes_{\mathbb{K}} \cdots \otimes_{\mathbb{K}} (\mathbb{K}^{n_{\pi(p)}})^\ast$,  the \emph{Kronecker product} of $T$ and $S$ is 
\[
T \boxtimes_{\mathbb{K}} S = (u_1\otimes v_1) \otimes \cdots \otimes (u_d \otimes v_d) \in \mathbb{K}^{n_1 m_1} \otimes_{\mathbb{K}} \cdots \otimes_{\mathbb{K}} \mathbb{K}^{n_d m_d}.
\]
We denote by $\langle T,  F \rangle $ or $\langle F,  T \rangle$ the tensor in $\mathbb{K}^{n_{\pi(p+1)}} \otimes_{\mathbb{K}} \cdots \otimes_{\mathbb{K}} \mathbb{K}^{\pi(n)}$ obtained by \emph{contracting $T$ with $F$}:
\[
\langle T,  S \rangle \coloneqq
\left( \prod_{i=1}^p f_i (u_{\pi(i)}) \right) u_{\pi(p+1)} \otimes \cdots \otimes u_{\pi(n)}. 
\]
The \emph{direct sum} of $T$ and $S$ is 
\[ 
T \oplus S \coloneqq (u_1,0)\otimes \cdots \otimes (u_d,0)+(0,v_1)\otimes \cdots \otimes (0,v_d) \in \mathbb{K}^{n_1 + m_1} \otimes_{\mathbb{K}} \cdots \otimes_{\mathbb{K}} \mathbb{K}^{n_d + m_d},
\]
where $(u,0)$ (resp.  $(0,v)$) denotes the image of $u \in \mathbb{K}^{n}$ (resp.  $v \in \mathbb{K}^{m}$) in $\mathbb{K}^{n + m}$ via the natural embedding. We extend these operations for general tensors by linearity.  
\subsection{Ranks of tensors}
For $S \in \mathbb{K}^{n_1} \otimes_{\mathbb{K}} \cdots \otimes_{\mathbb{K}} \mathbb{K}^{n_d}$,  if there exists some partition $\pi \in \mathfrak{S}_d$,  integer $1 \le p \le d-1$,  tensors $S_1 \in \mathbb{K}^{n_{\pi(1)}} \otimes_{\mathbb{K}} \cdots \otimes_{\mathbb{K}}  \mathbb{K}^{n_{\pi(p)}}$ and $S_2 \in \mathbb{K}^{n_{\pi(p+1)}} \otimes_{\mathbb{K}} \cdots \otimes_{\mathbb{K}} \mathbb{K}^{n_{\pi(d)}}$ such that
\[
S  = S_1 \otimes S_2,
\]
then we say $S$ has \emph{partition rank one}.  In particular,  if $p = 1$,  we say $S$ has \emph{slice rank one}.  Here the equality is understood via the natural isomorphism 
\[
\mathbb{K}^{n_1} \otimes_{\mathbb{K}} \cdots \otimes_{\mathbb{K}} \mathbb{K}^{n_d} \simeq 
\mathbb{K}^{n_{\pi(1)}} \otimes_{\mathbb{K}} \cdots \otimes_{\mathbb{K}}  \mathbb{K}^{n_{\pi(p)}}  \otimes_{\mathbb{K}}   \mathbb{K}^{n_{\pi(p+1)}} \otimes_{\mathbb{K}} \cdots \otimes_{\mathbb{K}}  \mathbb{K}^{n_{\pi(d)}}. 
\]
Given $T \in \mathbb{K}^{n_1} \otimes_{\mathbb{K}}  \cdots \otimes_{\mathbb{K}} \mathbb{K}^{n_d}$,  we denote 
\begin{align}
\pr_{\mathbb{K}}(T) &\coloneqq \min \left\lbrace
r \in \mathbb{N}: T = \sum_{j=1}^r T_j,\; T_j \text{~has partition rank one}
\right\rbrace,  \label{eq:PR} \\
\sr_{\mathbb{K}}(T) &\coloneqq \min \left\lbrace
r \in \mathbb{N}: T = \sum_{j=1}^r T_j,\; T_j \text{~has slice rank one}
\right\rbrace.  \label{eq:SR}
\end{align}

If $S  \in \mathbb{K}^{m_1} \otimes_{\mathbb{K}}  \cdots \otimes_{\mathbb{K}}  \mathbb{K}^{m_d}$ can be written as $S = (g_1 \otimes \cdots \otimes g_d) T$ for some $g_i \in \Hom(\mathbb{K}^{n_1}, \mathbb{K}^{m_1})$,  $1 \le i \le d$,  then we write $S \preceq T$.  We define \emph{cp-rank} and \emph{subrank} of $T$ respectively by
\begin{equation}\label{eq:R+Q}
\R_{\mathbb{K}}(T) \coloneqq \min \left\lbrace
r\in \mathbb{N}: T  \preceq  \Id_r
\right\rbrace, \quad
\Q_{\mathbb{K}}(T) \coloneqq \max \left\lbrace
s \in \mathbb{N}: \Id_s \preceq  T
\right\rbrace.
\end{equation}
Here $\Id_n \in \mathbb{K}^{n} \otimes_{\mathbb{K}} \cdots \otimes_{\mathbb{K}} \mathbb{K}^{n}$ denotes the identity tensor for each positive integer $n \le \min\{n_j\}_{j=1}^d$.  

The \emph{geometric rank} of $T$ is 
\begin{equation}\label{eq:GR}
\GR_{\mathbb{K}}(T) \coloneqq \sum_{i=1}^d  n_i  - n_k - \dim Z_k \left( T\right), 
\end{equation}
where  
\begin{equation}\label{eq:Z}
Z_k \left( T \right) \coloneqq \left\lbrace (f_1,\dots, f_{d-1}) \in  \bigoplus_{1 \le i \ne k \le d} \left( \mathbb{K}^{n_i} \right)^\ast:  \langle T,  f_1 \otimes \cdots \otimes f_{d-1} \rangle  = 0\right\rbrace.
\end{equation}
If $\mathbb{K}$ is a finite field,  we also define the \emph{analytic rank} of $T$:
\begin{equation}\label{eq:AR}
\AR_{\mathbb{K}}(T) \coloneqq -\log_{ |\mathbb{K}| } \frac{\sum_{(f_1,\dots, f_{d}) \in  \bigoplus_{1 \le i \le d} \left( \mathbb{K}^{n_i} \right)^\ast} \chi (\langle T,  f_1 \otimes \cdots \otimes f_{d} \rangle)}{|\mathbb{K}|^{n_1+ \cdots + n_d}},
\end{equation}
where $\chi: \mathbb{K} \to \mathbb{C}$ is any non-trivial additive character of $\mathbb{K}$.
\subsection{Properties of ranks}
For ease of reference,  we collect some basic properties of ranks defined in \eqref{eq:PR}--\eqref{eq:AR}.
\begin{lemma}[Additivity of slice rank]\cite{gowers2021slice}\label{lem4-3}
For any $T \in \mathbb{K}^{n_1} \otimes_{\mathbb{K}} \cdots \otimes_{\mathbb{K}} \mathbb{K}^{n_d}$ and $S \in \mathbb{K}^{m_1} \otimes_{\mathbb{K}}  \cdots \otimes_{\mathbb{K}} \mathbb{K}^{m_d}$,  we have $ \sr_{\mathbb{K}}(T\oplus S)=  \sr_{\mathbb{K}}(T)+ \sr_{\mathbb{K}}(S)$.
\end{lemma}

\begin{lemma}\cite{anal2019} \label{lem3-1}
For any $T\in \mathbb{F}_q^{n_1} \otimes_{\mathbb{F}_q} \cdots \otimes_{\mathbb{F}_q} \mathbb{F}_q^{n_d}$ and $1 \le k \le d$,  we have 
    \[
    \AR_{\mathbb{F}_{q}}(T)=\sum_{i=1}^{d} n_i - n_k -\log_{q}\left( \left\lvert  Z_k\left( T \right)  \right\rvert \right).
    \]
Here $Z_k\left( T \right)$ is defined by \eqref{eq:Z}.
\end{lemma}
\begin{lemma}[Additivity of analytic rank]\cite{anal2019}
For any $T\in \mathbb{F}_q^{n_1} \otimes_{\mathbb{F}_q} \cdots \otimes_{\mathbb{F}_q} \mathbb{F}_q^{n_d}$ and $S\in \mathbb{F}_q^{m_1} \otimes_{\mathbb{F}_q} \cdots \otimes_{\mathbb{F}_q} \mathbb{F}_q^{m_d}$,  we have $\AR_{\mathbb{F}_q}(T\oplus S)=\AR_{\mathbb{F}_q}(T)+\AR_{\mathbb{F}_q}(S)$.
    \label{lem3-4}
\end{lemma}

The following lemma is a well-known fact concerning the monotonicity of analytic rank with respect the restriction.  
\begin{lemma}[Monotonicity of analytic rank]\cite{anal2019,AJF22}
\label{lem3-3}
Given $T\in \mathbb{F}_q^{n_1} \otimes \cdots \otimes \mathbb{F}_q^{n_d}$ and $S \in \mathbb{F}_q^{m_1} \otimes \cdots \otimes \mathbb{F}_q^{m_d}$ such that $T\preceq S$,  we have $\AR_{\mathbb{F}_{q}}(T)\le \AR_{\mathbb{F}_{q}}(S)$.
\end{lemma}   
\section{Rank and subrank of a field extension}
Let $\mathbb{F}$ be  a finite extension of $\mathbb{K}$.  For an integer $d\ge 2$,  we consider the map 
\begin{equation}\label{eq:mult-tensor}
\varphi_d: \underbrace{\mathbb{F} \times \cdots \times \mathbb{F}}_{(d-1)\text{~copies}} \to \mathbb{F},\quad \varphi_d (x_1,\dots,  x_{d-1})  = x_1 \cdots x_{d-1}.
\end{equation}
As a $\mathbb{K}$-multilinear map,  $\varphi_d$ naturally corresponds to a tensor $\widetilde{M}_{d,\mathbb{F}} \in \underbrace{\mathbb{F}^{\ast} \otimes_{\mathbb{K}} \cdots  \otimes_{\mathbb{K}} \mathbb{F}^{\ast}}_{(d-1)\text{~copies}}  \otimes_{\mathbb{K}} \mathbb{F}$.  By choosing and fixing a $\mathbb{K}$-basis of $\mathbb{F}$,  we may identify $\mathbb{F}^{\ast}$ with $\mathbb{F}$ to obtain a tensor $M_{d,\mathbb{F}} \in  \underbrace{\mathbb{F} \otimes_{\mathbb{K}}  \cdots \otimes_{\mathbb{K}}  \mathbb{F}}_{d\text{~copies}}$.  Although $M_{d,\mathbb{F}}$ depends on the choice of $\mathbb{K}$-basis of $\mathbb{F}$,  the ranks of $M_{d,\mathbb{F}}$ discussed in this paper do not.  In fact,  the ranks of $M_{d,\mathbb{F}}$,  $\widetilde{T}_{d,\mathbb{F}}$ and $\varphi_d$ are all equal.  We identify $\mathbb{F}^\ast$ with $\mathbb{F}$ just to simplify notations.  All our arguments can be adjusted accordingly for $\widetilde{T}_{d,\mathbb{F}}$ and $\varphi_d$.  The main purpose of this section is to estimate $\R_{\mathbb{K}}(M_{d,\mathbb{F}})$ and $\Q_{\mathbb{K}}(M_{d,\mathbb{F}})$ defined by \eqref{eq:R+Q}.    
\begin{proposition}[Rank and subrank of field extension I]\label{lem2-1} 
Let $\mathbb{F}$ be a finite field extension of $\mathbb{K}$ and let $l \coloneqq  [\mathbb{F} : \mathbb{K}]$.  Suppose $d\ge 2$ is an integer.  
\begin{enumerate}[(a)]
\item If $\mathbb{K}$ is either an infinite perfect field or a finite field with $\lvert\mathbb{K}\rvert\ge (d-1)(l-1)+1$,  then 
\[
\R_{\mathbb{K}}(M_{d,\mathbb{F}}) \le (d-1)(l-1)+1.
\]
\item If $\mathbb{K}$ is either an infinite perfect field or a finite field,  then 
\[
\Q_{\mathbb{K}}(M_{d,\mathbb{F}}) \ge \max\left\lbrace \left\lceil \dfrac{l-1}{d-1}\right\rceil,  \lvert \mathbb{K} \rvert \right\rbrace.
\]
\end{enumerate}
\end{proposition}
\begin{proof}
Assume that $\mathbb{K}$ is infinite and perfect.  We have $\mathbb{F} = \mathbb{K}[\alpha]$ for some $\alpha \in \mathbb{F}$.  Thus,  we obtain the following commutative diagram:
\begin{equation}\label{eq:poly}
\begin{tikzcd}
	{\underbrace{\mathbb{K}[x]_{\le l-1} \times \cdots \times \mathbb{K}[x]_{\le l-1}}_{(d-1)\text{~copies}}} & {\mathbb{K}[x]_{\le (d-1)(l-1)}} \\
	{\underbrace{\mathbb{K}[\alpha] \times \cdots \times \mathbb{K}[\alpha]}_{(d-1)\text{~copies}}} & {\mathbb{K}[\alpha]}
	\arrow["{\Phi_d}", from=1-1, to=1-2]
	\arrow["{q^{-1} \times \cdots \times q^{-1}}"', from=2-1, to=1-1]
	\arrow["q"', from=1-2, to=2-2]
	\arrow["{\varphi_d}", from=2-1, to=2-2]
\end{tikzcd}
\end{equation}
where $\Phi_d$ is the map multiplying $(d-1)$ polynomials of degree at most $(l-1)$ and $q$ is the restriction of the quotient map $\mathbb{K}[x] \to \mathbb{K}[\alpha]$.  This implies that $\R_{\mathbb{K}}(M_{d,\mathbb{F}}) \le \R_{\mathbb{K}}(M_d)$ where $M_d \in (\mathbb{K}^{l})^{\otimes (d-1)} \otimes  \mathbb{K}^{(d-1)(l-1) + 1}$ is the tensor corresponding to $\Phi_d$.  Hence it is sufficient to estimate $\R_{\mathbb{K}}(M_d)$.  To that end,  we denote $N \coloneqq (d-1)(l-1)$ and pick distinct elements $a_{1},\dots,a_{N+1}$ in $\mathbb{K}$.  For each $1 \le j \le N+1$ and $1 \le k \le d$,  we define
\[
\eta_{jk} \coloneqq 
\begin{cases}
\sum_{i=0}^{l-1} a_j^i \varepsilon_i &\text{~if~} 1 \le k \le d-1 \\ 
\sum_{i=0}^{N} (V^{-1})_{ji} \varepsilon_i &\text{~if~} k = d \\
\end{cases}
\]
where $(V^{-1})_{ji}$ is the $(j,i)$-th element of the inverse of the Vandermonde matrix   
\[ 
V = \begin{bmatrix}
1 & a_1 & \cdots & a_1^{N} \\
\vdots & \vdots & \ddots & \vdots \\
1 & a_{N+1} & \cdots & a_{N+1}^{N} \\
\end{bmatrix} \in \mathbb{K}^{(N+1) \times (N+1)}
\]
and $\varepsilon_s: \mathbb{K}[x]  \to \mathbb{K}$ is the linear function such that $\varepsilon_s(x^t) = \delta_{st}$,  $0 \le s,  t\le N $. 
We consider 
\[
M'_d \coloneqq \sum_{j=1}^{N+1} \eta_{j1} \otimes \cdots \otimes \eta_{jd}.
\] 
Given $f_1,\dots,  f_{d-1} \in \mathbb{K}[x]_{\le l-1}$,  it is clear that $M'_d(f_1,\dots,  f_{d-1})$ computes the product polynomial $\prod_{k=1}^{d-1} f_k \in \mathbb{K}[x]_{\le N}$ by interpolation at points $a_{1},\dots,a_{N+1} \in \mathbb{K}$. Thus $\R_{\mathbb{K}}(M_d) = \R_{\mathbb{K}}(M'_d) \le N+1$.

For $\Q_{\mathbb{K}}(M_{d,\mathbb{F}})$,  we let $m \coloneqq \lfloor (l-1)/(d-1) \rfloor$ and consider the commutative diagram:
\begin{equation}\label{eq:poly1}
\begin{tikzcd}
	{\underbrace{\mathbb{K}[\alpha] \times \cdots \times \mathbb{K}[\alpha]}_{(d-1)\text{~copies}}} & {\mathbb{K}[\alpha]} \\
	{\underbrace{\mathbb{K}[x]_{\le m} \times \cdots \times \mathbb{K}[x]_{\le m}}_{(d-1)\text{~copies}}} & {\mathbb{K}[x]_{\le l-1}}
	\arrow["{\varphi_d}", from=1-1, to=1-2]
	\arrow["{q^{-1}}", from=1-2, to=2-2]
	\arrow["{q\times \cdots \times q}"', from=2-1, to=1-1]
	\arrow["{\Psi_d}", from=2-1, to=2-2]
\end{tikzcd}
\end{equation}
where $\Psi_d$ denotes the map multiplying $(d-1)$ polynomials of degree at most $m$.  This implies that $\Q_{\mathbb{K}}(M_{d,\mathbb{F}}) \ge \Q_{\mathbb{K}}(\Psi_d)$.  For distinct $b_{1},\cdots,b_{m+1} \in \mathbb{K}$,  we define a linear map
\[
S:\mathbb{K}[x]\to\mathbb{K}^{m+1},\quad S(f) \coloneqq (f(b_{1}),\cdots,f(b_{m+1})). 
\]
By interpolation,  it is straightforward to verify that $S\mid_{\mathbb{K}[x]_{\le m}}$ is bijective.  This leads to
\[\begin{tikzcd}
	{\underbrace{\mathbb{K}[x]_{\le m} \times \cdots \times \mathbb{K}[x]_{\le m}}_{(d-1)\text{~copies}}} & {\mathbb{K}[x]}_{\le l-1} \\
	{\underbrace{\mathbb{K}^{m+1} \times \cdots \times \mathbb{K}^{m+1}}_{(d-1)\text{~copies}}} & {\mathbb{K}^{m+1}}
	\arrow["{\Psi_d}", from=1-1, to=1-2]
	\arrow["{S}", from=1-2, to=2-2]
	\arrow["{S^{-1}\times \cdots \times S^{-1}}"', from=2-1, to=1-1]
	\arrow["{\psi_d}", from=2-1, to=2-2]
\end{tikzcd}\]
where $\psi_d$ is defined as $\psi_d(x_1,\dots,  x_{d-1}) \coloneqq  x_1 \cdots x_{d-1}$ for the trivial $\mathbb{K}$-algebra $\mathbb{K}^{m+1}$.  Since the tensor corresponding to $\psi_d$ is the identity tensor $\Id_{m+1}$,  we obtain $\Q_{\mathbb{K}}(M_{d,\mathbb{F}}) \ge \Q_{\mathbb{K}}(\Psi_d) \ge m+1$.

When $\mathbb{K}$ is a finite field with $\vert\mathbb{K}\vert\ge (d-1)(l-1)+1$ (resp. $\vert\mathbb{K}\vert\ge \lfloor (l-1)/(d-1) \rfloor +1$),  the above argument for $\R_{\mathbb{K}} (M_{d,\mathbb{F}})$ (resp.  $\Q_{\mathbb{K}} (M_{d,\mathbb{F}})$) applies verbatim.  For $\vert\mathbb{K}\vert \le  \lfloor (l-1)/(d-1) \rfloor$,  we set $m \coloneqq \vert\mathbb{K}\vert$ in the argument for $\Q_{\mathbb{K}} (M_{d,\mathbb{F}})$ to ensure the interpolation remains applicable.
\end{proof}
If $\mathbb{K} = \mathbb{F}_q$ is a finite field,  we record the following estimate of $\R_{\mathbb{F}_q}(M_{d,\mathbb{F}})$ for easy reference.
\begin{theorem}[Rank of field extension II] 
\cite[Theorem~3]{Ballet2022}\label{thm2-2}
Given any integer $d > 0$ and prime power $q$,  there is a constant $C \coloneqq C(d,q) > 0 $ such that for any finite field extension $\mathbb{F}/\mathbb{F}_q$,  we have $\R_{\mathbb{F}_q}( M_{d,\mathbb{F}} )\le C l $ where $l \coloneqq [\mathbb{F}: \mathbb{F}_q]$.  
\end{theorem}
We remark that if $d  = 3$,  Theorem~\ref{thm2-2} improves a classical result in algebraic complexity theory \cite[Theorem~2]{LSW83}.  Next we consider $\Q_{\mathbb{K}}( M_{d,\mathbb{F}} )$ when $\mathbb{K}$ is a finite field.  We borrow the idea behind the Chudnovsky-Chudnovsky method \cite{CC87,Shokrollahi92,  ballet1999curves,  Ballet2022} designed to obtain a sharp upper bound of $\R_{\mathbb{K}}( M_{3,\mathbb{F}})$. To achieve the goal,  we need to establish some technical lemmas.
\begin{lemma} \label{lem2-8}
Let $\mathbb{F}/\mathbb{F}_{q}$ be an algebraic function field with genus $g$ and let $N,  d, n$ be positive integers such that $d\ge 2$, $N \ge g + 1 $,  $(d-1) (N + g-1) < n$.  If $\mathbb{F}/\mathbb{F}_{q}$ has at least $N$ degree one places $\mathfrak{m}_{1},\dots,  \mathfrak{m}_{N}$ and a degree $n$ place $\mathfrak{m}$,  then there exists a divisor $D$ such that
\begin{enumerate}[(a)]  
\item $v_{\mathfrak{m}}(D) = v_{\mathfrak{m}_1}(D) = \cdots = v_{\mathfrak{m}_N}(D) = 0$.
\item The map \begin{equation}\label{eq:E}
E:\mathscr{L}(D) \to  \mathbb{F}_{q}^{N},\quad 
        E(f) = (f(\mathfrak{m}_{1}),\cdots,f(\mathfrak{m}_{N}))
\end{equation}
is bijective.
\item The map \begin{equation}\label{eq:S}
S:\mathscr{L}((d-1) D) \to \mathcal{O}/\mathfrak{m},\quad S(f) = f(\mathfrak{m})
\end{equation}    
is injective. 
\end{enumerate}
Here $\mathcal{O}$ is the valuation ring associated to the place $\mathfrak{m}$.
\end{lemma}
\begin{proof}
By Lemma~\ref{lem2-7},  there exists a non-special divisor $D_0$ with $\deg(D_0)=g-1$.  Let $D_{1} \coloneqq D_0 + \mathfrak{m}_{1}+\cdots+ \mathfrak{m}_{N}$.  According to Lemma~\ref{lem2-5},  there exists $f_0 \in \mathbb{F}$ such that $v_{\mathfrak{m}_{i}}((f_0)+ D_{1})=0$ and $v_{\mathfrak{m}}((f_0)+D_{1})=0$ for $1 \le i \le N$.  Denote $D \coloneqq (f_0) + D_{1}$. Then $\deg(D)= \deg( D_{1})=N+g-1\ge 2g$ by Lemma~\ref{lem2-4}.  Thus,  Lemma~\ref{lem2-3} implies that $\ell(D) = \deg(D)+1-g=N$. 

The map $E$ is well-defined since for each $1 \le i \le N$,  $v_{\mathfrak{m}_{i}}(D) = v_{\mathfrak{m}_{i}}((f_0)+ D_1)=0$.  Moreover,  we observe that 
\[
\Ker(E) = \mathscr{L}(D-\mathfrak{m}_{1}-\cdots-\mathfrak{m}_{N}) = \mathscr{L}( (f_0) + D_0 ).
\]
Since $D_0$ is non-special, we have 
\[
\dim \Ker (E) = \ell( (f_0) + D_0 )= \ell (D_0)= \deg(D_0)+1-g=0.
\]
Therefore,  $E$ is an isomorphism.

Since $v_{\mathfrak{m}}(D)=0$, the map $S$ is also well-defined and 
\[
\Ker(S) = \mathscr{L}((d-1)D-\mathfrak{m}).
\]    
By assumption, $\deg((d-1)D-\mathfrak{m})=d(N + g-1)-n< 0$.  Hence $\ell((d-1)D-\mathfrak{m})=0$ by Lemma~\ref{lem2-6}. 
\end{proof}

\begin{lemma}\label{thm2-9}
Let $n,  d \ge 2,  N$ be positive integers.  Suppose $\mathbb{F}/\mathbb{F}_{q}$ is an algebraic function field of genus $g$ such that 
\begin{enumerate}[(a)]
        \item $N\ge g+1$ and $(d-1) (N + g-1) < n$.
        \item $\mathbb{F}/\mathbb{F}_{q}$ has at least $N$ degree one places and one degree $n$ place.
\end{enumerate}
Then $\Q_{\mathbb{F}_q}( M_{d,\mathbb{F}_{q^n}} ) \ge N$.
\end{lemma}
\begin{proof}
Let $\mathfrak{m}_1,\dots, \mathfrak{m}_N$ be distinct degree one places and let $\mathfrak{m}$ be a degree $n$ place in Lemma~\ref{lem2-8}. We denote by $\mathcal{O}$ (resp.  $\mathcal{O}_1,\dots,  \mathcal{O}_N$) the discrete valuation ring determined by $\mathfrak{m}$ (resp.  $\mathfrak{m}_1,\dots,  \mathfrak{m}_N$).  By Lemma~\ref{lem2-8}, there exists a divisor $D$ such that $\ell(D) = N$ and the $\mathbb{F}_q$-linear map $E$ (resp.  $S$) defined in \eqref{eq:E} (resp. \eqref{eq:S}) is bijective (resp.  injective).  Since $S$ is injective,  there is a $\mathbb{F}_q$-linear map $S^{-1}: \mathcal{O}/\mathfrak{m} \to \mathscr{L}((d-1)D)$ such that $S^{-1} \circ S = \id_{\mathscr{L}((d-1)D)}$. 

We consider the following diagram:
\begin{equation}\label{eq:thm2-9}
\begin{tikzcd}
	{\underbrace{\mathbb{F}_q^N\times \cdots \times \mathbb{F}_q^N}_{(d-1)\text{~copies}}} & {\mathbb{F}_q^N} \\
	{\underbrace{ \mathcal{O}/\mathfrak{m} \times \cdots \times \mathcal{O}/\mathfrak{m}}_{(d-1)\text{~copies}}} & {\mathcal{O}/\mathfrak{m}}
	\arrow["{\varphi_d}", from=1-1, to=1-2]
	\arrow["{(\widetilde{S} \circ  E^{-1}) \times \cdots \times (\widetilde{S} \circ E^{-1}) }"', from=1-1, to=2-1]
	\arrow["{\psi_d}"', from=2-1, to=2-2]
	\arrow["{\widetilde{E}\circ S^{-1}}"', from=2-2, to=1-2]
\end{tikzcd}
\end{equation}
Here $\widetilde{S}: \mathcal{O} \to \mathcal{O}/{\mathfrak{m}} \simeq \mathbb{F}_{q^N}$ and $\widetilde{E}:  \mathcal{O}_1 \cap\cdots\cap\mathcal{O}_N \to \mathbb{F}_{q}^{N}$ are evaluation maps defined by
$\widetilde{S}(f) = f(\mathfrak{m})$ and 
$\widetilde{E}(f) = (f(\mathfrak{m}_1),\dots,  f(\mathfrak{m}_N))$,  respectively.  We claim that vertical maps in  \eqref{eq:thm2-9} are well-defined and the diagram commutes.  This implies that 
\[
\Q_{\mathbb{F}_q} (M_{d,\mathbb{F}_{q^n}}) = \Q_{\mathbb{F}_q} (M_{d,\mathcal{O}/\mathfrak{m}}) \ge \Q_{\mathbb{F}_q} (M_{d,\mathbb{F}_q^N})=N.
\] 

It is left to prove the claim.  We first verify that vertical maps are well-defined.  According to Lemma~\ref{lem2-8},  we have $v_{\mathfrak{m}}(D) = v_{\mathfrak{m}_1}(D) = \cdots = v_{\mathfrak{m}_N}(D)=0$.  This implies that $v_{\mathfrak{m}}(f) \ge 0$ if $f \in \mathcal{L}(D)$ and $v_{\mathfrak{m}_j}(f) \ge 0$ for all $1 \le j \le N$ if $f \in \mathcal{L}((d-1)D)$,  which lead to $\mathcal{L}(D) \subseteq \mathcal{O}$ and $\mathcal{L}((d-1)D) \subseteq \bigcap_{j=1}^N \mathcal{O}_j$.  Therefore,  maps $\widetilde{S} \circ E^{-1}$ and $\widetilde{E} \circ S^{-1}$ are well-defined.

Next,  we prove the commutativity of \eqref{eq:thm2-9}.  For any $a_1,\dots,  a_{d-1} \in \mathbb{F}_q^N$,  we denote $f_j \coloneqq E^{-1}(a_j) \in \mathscr{L}(D)$,  $1 \le j \le d-1$.  Thus 
\[
\psi_d( \widetilde{S}(f_1),\dots,  \widetilde{S}(f_{d-1}))) = \prod_{j=1}^{d-1} \widetilde{S}(f_j) = \widetilde{S} \left( \prod_{j=1}^{d-1} f_j\right).
\]
Here the second equality holds since $\widetilde{S}$ is an $\mathbb{F}_q$-algebra homomorphism.  We observe that $\prod_{j=1}^{d-1} f_j \in \mathcal{L}((d-1)D)$.  This implies that 
\[
(\widetilde{E} \circ S^{-1}) \circ  \psi_d \circ (\widetilde{S} \circ E^{-1}) (a_1, \dots,  a_{d-1})= \widetilde{E} \left(\prod_{j=1}^{d-1}f_{j}\right).
\]
Since $\widetilde{E}$ is an $\mathbb{F}_{q}$-algebra homomorphism, we have,
\[
\widetilde{E} \left(\prod_{j=1}^{d-1}f_{j}\right)=\prod_{j=1}^{d-1} \widetilde{E} (f_{j})=\prod_{j=1}^{d-1}a_{i}=\varphi_d(a_1,\dots, a_{d-1}). 
\]
\end{proof}
\begin{lemma}\label{Poly-finite} 
If $d \ge 2$ and $(l-1)\ge (d-1)(n-1)$,  then we have $\Q_{\mathbb{F}_q}(M_{d,\mathbb{F}_{q^{n}}}) \le \Q_{\mathbb{F}_q}(M_{d,\mathbb{F}_{q^{l}}})$.
\end{lemma}
\begin{proof}
Since $\mathbb{F}_q$ is a perfect field,  we observe that \eqref{eq:poly} (resp.  \eqref{eq:poly1}) holds for $(\mathbb{K},  \mathbb{F})= (\mathbb{F}_q,\mathbb{F}_{q^n})$ (resp.  $(\mathbb{K},  \mathbb{F})= (\mathbb{F}_q,\mathbb{F}_{q^l})$). Thus,  we have $M_{d,\mathbb{F}_{q^n}} \preceq M_{d,n,\mathbb{F}_q} \preceq M_{d,\mathbb{F}_{q^l}}$,  from which we obtain $Q_{\mathbb{F}_q}(M_{d,\mathbb{F}_{q^n}}) \le   Q_{\mathbb{F}_q}(M_{d,\mathbb{F}_{q^l}})$.  Here $M_{d,n,\mathbb{F}_q}$ is the tensor corresponding to the multilinear map multiplying $(d-1)$ $\mathbb{F}_q$-polynomials of degree at most $(n-1)$:
\[
\Phi_d:\underbrace{\mathbb{F}_q [x]_{\le n-1} \times \cdots \times \mathbb{F}_q [x]_{\le n-1}}_{(d-1) \text{~copies}} \to \mathbb{F}_q[x]_{(d-1)(n-1)},\quad \Phi_d(f_1,\dots, f_{d-1}) = \prod_{j=1}^{d-1}f_j. 
\]
\end{proof}
\begin{lemma}\label{lem2-12} 
For positive integers $m,n$ and $d \ge 2$,  we have $\Q_{\mathbb{F}_q}(M_{d,\mathbb{F}_{q^{mn}}})\ge \Q_{\mathbb{F}_q}(M_{d,\mathbb{F}_{q^{m}}})
    \Q_{\mathbb{F}_{q^m}}(M_{d,\mathbb{F}_{q^{mn}}})$.
\end{lemma}
\begin{proof}
We notice that  
\[
M_{d,\mathbb{F}_{q^{m}}} \in \underbrace{\mathbb{F}_{q^m} \otimes_{\mathbb{F}_q}  \cdots \otimes_{\mathbb{F}_q} \mathbb{F}_{q^m}}_{d\text{~copies}},\quad M_{d,\mathbb{F}_{q^{mn}}} \in \underbrace{\mathbb{F}_{q^{mn}} \otimes_{\mathbb{F}_{q^m}}  \cdots \otimes_{\mathbb{F}_{q^m}} \mathbb{F}_{q^{mn}}}_{d\text{~copies}}
\]
and 
\[
M_{d,\mathbb{F}_{q^{m}}} \boxtimes_{\mathbb{F}^{m}_q} M_{d,\mathbb{F}_{q^{mn}}} = M_{d,\mathbb{F}_{q^{mn}}} \in \underbrace{\mathbb{F}_{q^{mn}} \otimes_{\mathbb{F}_{q}}  \cdots \otimes_{\mathbb{F}_{q}} \mathbb{F}_{q^{mn}}}_{d\text{~copies}},
\]
We recall that for tensors $A \preceq  B$ and $C \preceq D$,  it holds that $A\boxtimes C \preceq B\boxtimes D$.  Thus,  the desired inequality follows immediately from the definition of the subrank. 
\end{proof}  
\begin{theorem}[Subrank of field extension II]
For any positive integers $d\ge 2,n$ and prime power $q$,  we have $\Q_{\mathbb{F}_q} (M_{d, \mathbb{F}_{q^n}}) \ge \frac{n}{96d^2q}$.
\label{thm2-13}
\end{theorem}
\begin{proof}
We denote $\sigma_q(n)  \coloneqq \Q_{\mathbb{F}_q} (M_{d, \mathbb{F}_{q^n}})$ and we first estimate $\sigma_{q^2}(n)$.  For each positive integer $k$,  we let $g_k$ and $N_k$ (if $k \ge 3$) be positive integers defined as in Lemma~\ref{lem2-10}.  We consider 
\[
m \coloneqq \max\{k \in \mathbb{N}: n \ge 2 d g_k + 3 \}
\]
and for $m \ge 3$
\[
N \coloneqq \max\{N'\in \mathbb{N}: n > d(N' + g_m - 1), \; N_m \ge N' \}.
\]
If $n \ge 2 d g_3 + 3$,  then $m\ge 3$.  By Lemma~\ref{lem2-10},  there exists an algebraic field extension $\mathbb{F}/\mathbb{F}_{q^2}$ of genus $g_m$ admitting at least $N_m$ degree one places.  Since $n \ge 2 d g_m + 3 \ge 4 g_m + 3$,  we have $N \ge g_m + 1$ and Lemma~\ref{lem2-11} implies that $\mathbb{F}/\mathbb{F}_{q^2}$ admits a place of degree $n$.  Lemma~\ref{thm2-9} ensures that 
\[
\sigma_{q^2}(n) \ge N.
\]
By \eqref{eq:Nk},  we have $N_m \ge g_m$.  The maximality of $m$ and $N$ leads to
\[
2 d g_{m+1} + 3 > n,\quad N + 1 > \min\left\lbrace \frac{n}{d} - g_{m},  N_{m}\right\rbrace \ge\min \{ g_{m},N_{m} \} = g_{m}
\]
According to  \eqref{equ3},  we obtain by a direct calculation that $g_{m+1} < 2qg_m$,  from which we have 
\[
N > g_m - 1 > \frac{g_{m+1}}{2q}  - 1 > \frac{n-3}{4qd} - 1 \ge \frac{n}{24qd}.
\]  
The last inequality can be verified as follows. This inequality is equivalent to $5n\ge 18+24qd$.  Since $n\ge 2dg_{3}+3=2d(q^{3}-2q+1)+3\ge 5dq+3$, we have $5n\ge 5(5dq+3)\ge 24dq+18$.
For $n<2dg_{3}+3\le 2dq^{3}$,  we split the discussion into two cases.  If $q^{2}\ge n/d$, Proposition~\ref{lem2-1} implies that $\sigma_{q^{2}}(n)\ge n/d$.  If $q^{2}< n/d$,  a similar interpolation argument as in the proof of Proposition~\ref{lem2-1} shows $\sigma_{q^{2}}(n)\ge q^{2}$.  Therefore,  we obtain $\sigma_{q^{2}}(n)\ge \min(n/d,q^{2})\ge\frac{1}{2qd}n$.  To complete the proof,  we recall from Lemma~\ref{lem2-12} that
\[
\sigma_q(2n) \ge \sigma_q (2) \sigma_{q^2}(n) \ge \frac{n}{24qd}. 
\]
Given an odd integer $l > 2d$,  we let $n \coloneqq \left\lfloor \frac{1}{2} \left( \frac{l-1}{d-1} + 1 \right) \right\rfloor$ so that 
\[
l-1 \ge (d-1)(2n-1),\quad n \ge \frac{1}{2} \left( \frac{l-1}{d-1} - 1 \right) > \frac{l}{4d}.   
\] 
Lemma~\ref{Poly-finite} implies $\sigma_q(l) \ge \sigma_q(2n) \ge n/24qd >  l/ 96qd^2$ and this completes the proof.
\end{proof}
 
\section{Analytic rank and Geometric rank under field extension}
This section is devoted to a discussion of the behaviour of $\AR_{\mathbb{K}}(T)$ and $\GR_{\mathbb{K}}(T)$ under field extensions,  defined respectively in \eqref{eq:AR} and \eqref{eq:GR}.  For each $T \in \mathbb{F}_q^{n_1} \otimes \cdots \otimes \mathbb{F}_q^{n_d}$ and a field extension $\mathbb{F}/\mathbb{F}_q$,  we denote by $T^{\mathbb{F}}$ the tensor in $\mathbb{F}^{n_1} \otimes_{\mathbb{F}} \cdots \otimes_{\mathbb{F}} \mathbb{F}^{n_d}$ determined by $T$.  To be more specific,  we regard $T$ (resp.  $T^{\mathbb{F}}$) as an $\mathbb{F}_q$-multilinear function (resp.  $\mathbb{F}$-multilinear function) on $(\mathbb{F}_q^{n_1})^\ast \times \cdots \times (\mathbb{F}_q^{n_d})^\ast$ (resp.  $(\mathbb{F}^{n_1})^\ast \times \cdots \times (\mathbb{F}^{n_d})^\ast$).  Then $T^{\mathbb{F}}$ is obtained by extending $T$ linearly over $\mathbb{F}$.  For instance, 
\[
T^{\mathbb{F}} \left( a_1 x_1,\dots,  a_d x_d \right) \coloneqq a_1 \cdots a_d T(x_1 ,\dots,  x_d)
\]
where $a_1,\dots,  a_d \in \mathbb{F}$ and $x_1 \in (\mathbb{F}_q^{n_1})^\ast, \dots,  x_d \in (\mathbb{F}_q^{n_d})^\ast$.  

We also recall that there is a canonical $\mathbb{F}$-isomorphism
\[
\mathbb{F}^{n_{1}}\otimes_{\mathbb{F}}\cdots\otimes_{\mathbb{F}}\mathbb{F}^{n_{d}} \simeq \operatorname{Hom}_{\mathbb{F}}(  (\mathbb{F}^{n_{1}})^\ast \times \cdots \times  (\mathbb{F}^{n_{d-1}})^\ast,  \mathbb{F}^{n_d} ),
\]
where $\operatorname{Hom}_{\mathbb{F}}(  (\mathbb{F}^{n_{1}})^\ast \times \cdots \times  (\mathbb{F}^{n_{d-1}})^\ast,  \mathbb{F}^{n_d} )$ is the vector space consisting of all $\mathbb{F}$-multilinear maps from $(\mathbb{F}^{n_{1}})^\ast \times \cdots \times  (\mathbb{F}^{n_{d-1}})^\ast$ to $\mathbb{F}^{n_d}$.  We define the $\mathbb{F}_q$-linear map:
\begin{equation}\label{eq:iota}
\iota: \operatorname{Hom}_{\mathbb{F}}(  (\mathbb{F}^{n_{1}})^\ast \times \cdots \times  (\mathbb{F}^{n_{d-1}})^\ast,  \mathbb{F}^{n_d} ) \to \operatorname{Hom}_{\mathbb{F}_q}(  (\mathbb{F}^{n_{1}})^\ast \times \cdots \times  (\mathbb{F}^{n_{d-1}})^\ast,  \mathbb{F}^{n_d} ),  \quad \iota(L) = L.
\end{equation} 
As an $\mathbb{F}_q$-multilinear map,  we define $S_{\mathbb{F}_q} \coloneqq \iota (S)$.  In the context of algebraic complexity theory \cite[Chapter~15.3]{BCS97},  $T^{\mathbb{F}}$ and $S_{\mathbb{F}_q}$ are called the scalar extension and restriction of $T$ and $S$,  respectively. 
\begin{remark}\label{rmk:digression}
As a digression,  we briefly discuss an alternative definition of $S_{\mathbb{F}_q}$ in terms of the trace map.  Suppose that $\mathbb{F}$ is a separable extension of $\mathbb{F}_q$.  We denote by $\operatorname{tr}: \mathbb{F}\to \mathbb{F}_q $ a trace map such that $\operatorname{tr}(c) = c$ for any $c\in \mathbb{F}_q$.  Such a map exists since $\mathbb{F}$ is a separable extension of $\mathbb{F}_q$ (cf.  \cite[Proof of Corollary~1]{cohen2023partition}).  Given an $\mathbb{F}$-tensor $S \in \mathbb{F}^{n_{1}}\otimes_{\mathbb{F}}\cdots\otimes_{\mathbb{F}}\mathbb{F}^{n_{d}}$,  we may regard it either as an $\mathbb{F}$-multilinear map $\widetilde{S}: (\mathbb{F}^{n_1})^\ast \times \cdots \times (\mathbb{F}^{n_{d-1}})^\ast \to \mathbb{F}^{n_d}$ or as an $\mathbb{F}$-multilinear function $\overline{S}: (\mathbb{F}^{n_1})^\ast \times \cdots \times (\mathbb{F}^{n_d})^\ast \to \mathbb{F}$.  Canonically,  $\widetilde{S}$ and $\overline{S}$ are related by
\[
\langle z_d,  \widetilde{S} (z_1,\dots,  z_{d-1}) \rangle = \overline{S}(z_1,\dots,  z_d),\quad (z_1,\dots,  z_d) \in (\mathbb{F}^{n_1})^\ast \times \cdots \times (\mathbb{F}^{n_d})^\ast.
\]
Throughout this paper,   we do not distinguish between $S$,  $\widetilde{S}$ and $\overline{S}$,  and we use $S$ to denote all of them.  

Accordingly,  $S_{\mathbb{F}_q}$ can also defined as an $\mathbb{F}_q$-multilinear function
\[
S_{\mathbb{F}_q}: (\mathbb{F}^{n_1})^\ast \times \cdots \times (\mathbb{F}^{n_d})^\ast \to \mathbb{F}_q,\quad  S_{\mathbb{F}_q}(z_1,\dots,  z_d)\coloneqq   \tr (S (z_1,\dots,  z_d)).
\]  
We need to verify that $\iota (S) = {\tr} \circ S$,  when both $S$ and $\iota(S)$ are regarded as multilinear functions.  Since both sides are $\mathbb{F}_q$-multilinear,  it is sufficient to verify the equality on an $\mathbb{F}_q$-basis of $(\mathbb{F}^{n_{1}})^\ast \times \cdots \times  (\mathbb{F}^{n_{d-1}})^\ast$.  For each $1 \le k \le d$,  let $\{z^{(k)}_{1},  \dots,  z^{(k)}_{n_k}\}$ be an $\mathbb{F}$-basis of $(\mathbb{F}^{n_k})^\ast$,  and let $\{x^{(k)}_{1},  \dots,  x^{(k)}_{n_k}\}$ be the $\mathbb{F}$-basis of $\mathbb{F}^{n_k}$ dual to $\{z^{(k)}_{1},  \dots,  z^{(k)}_{n_k}\}$.  Suppose that $\{\beta_1, \dots,  \beta_r\}$ is an $\mathbb{F}_q$-basis of $\mathbb{F}$.  Therefore,  $\{ \beta_i x^{(d)}_{j}: 1 \le i \le r,\; 1 \le j \le n_d \}$ is an $\mathbb{F}_q$-basis of $\mathbb{F}^{n_d}$.  Let $\{ f_{ij}: 1 \le i \le r,\; 1 \le j \le n_d \}$ be the $\mathbb{F}_q$-basis of $(\mathbb{F}^{n_d})^\ast$ dual to $\{ \beta_i x^{(d)}_{j}: 1 \le i \le r,\; 1 \le j \le n_d \}$.  We suppose
\[
\iota(S)(\beta_{i_1} z^{(1)}_{j_1},  \dots \beta_{i_{d-1}} z^{(d-1)}_{j_{d-1}}) = \sum_{i_d, j_d = 1}^{r,n_d} C^{i_1,\dots, i_d}_{j_1,\dots, j_{d}} \left( \beta_i x_j^{(d)} \right) \in \mathbb{F}^{n_d},
\]
where $C^{i_1,\dots, i_d}_{j_1,\dots, j_{d}} \in \mathbb{F}_q$.  By definition we also have  
\[
S(\beta_{i_1} z^{(1)}_{j_1},  \dots \beta_{i_{d-1}} z^{(d-1)}_{j_{d-1}}) =  \sum_{i_d, j_d = 1}^{r,n_d} \beta_{i_d}  C^{i_1,\dots, i_d}_{j_1,\dots, j_{d}} x_{j_d}^{(d)} \in \mathbb{F}^{n_d}.
\]
Therefore,  we obtain 
\begin{align*}
\operatorname{tr} \left( \langle S(\beta_{i_1} z^{(1)}_{j_1},  \dots \beta_{i_{d-1}} z^{(d-1)}_{j_{d-1}}),  f_{ij}\rangle_{\mathbb{F}} \right) &=\sum_{i_d, j_d = 1}^{r,n_d} C^{i_1,\dots, i_d}_{j_1,\dots, j_{d}}  \operatorname{tr} \left(  \langle \beta_{i_d} x^{(d)}_{j_d},  f_{ij} \rangle_{\mathbb{F}} \right) \\
&= C^{i_1,\dots, i_{d-1},i}_{j_1,\dots, j_{d-1},j} \\
&= \langle \iota(S)(\beta_{i_1} z^{(1)}_{j_1},  \dots \beta_{i_{d-1}} z^{(d-1)}_{j_{d-1}}),   f_{ij} \rangle_{\mathbb{F}_q}.
\end{align*} 
Here $\langle x,  f \rangle_{\mathbb{F}}$ (resp.  $\langle x,  f \rangle_{\mathbb{F}_q}$) denotes the natural paring of $x\in \mathbb{F}^{n_d}$ with $f\in (\mathbb{F}^{n_d})^\ast$ over $\mathbb{F}$ (resp.  $\mathbb{F}_q$). 
\end{remark}
Let $M_{d,\mathbb{F}} \in \underbrace{\mathbb{F} \otimes_{\mathbb{F}_q}  \cdots \otimes_{\mathbb{F}_q}  \mathbb{F}}_{d\text{~copies}}$ be the tensor corresponding to the multiplication map $\varphi_d$ defined in \eqref{eq:mult-tensor}.  Suppose that $T \in \mathbb{F}_q^{n_1} \otimes_{\mathbb{F}_q} \cdots \otimes_{\mathbb{F}_q} \mathbb{F}_q^{n_d}$ is an $\mathbb{F}_q$-tensor.  We recall from \cite[Equation~(15.14)]{BCS97} that
\[ 
(T^{\mathbb{F}})_{\mathbb{F}_q} =T\boxtimes_{\mathbb{F}_q} M_{d,\mathbb{F}} \in \mathbb{F}^{n_{1}}\otimes_{\mathbb{F}_q}\cdots\otimes_{\mathbb{F}_q}\mathbb{F}^{n_{d}}.   
\]
To simplify notation,  we use $T^{\mathbb{F}}$ to denote both $T^{\mathbb{F}}$ and $(T^{\mathbb{F}})_{\mathbb{F}_q}$ when there is no risk of confusion. 
\begin{theorem}[Stability of analytic/geometric rank]\label{thm3-5}
Given an integer $d > 0$ and prime power $q$,  there exist constant numbers $c \coloneqq c (d,q) > 0$ and $C \coloneqq C (d,q) > 0$ such that for any positive integers $n_1,\dots , n_d$ and tensor $T\in \mathbb{F}_q^{n_1} \otimes \cdots \otimes \mathbb{F}_q^{n_d}$,  we have  
\[
c \AR_{\mathbb{F}_{q}}(T)\le \AR_{\mathbb{F}}\left( T^{\mathbb{F}} \right) \le C \AR_{\mathbb{F}_{q}}(T)
\]
for any finite extension of $\mathbb{F}/\mathbb{F}_q$.  Moreover,  we also have $\GR_{\mathbb{F}_q}(T) = \GR_{\overline{\mathbb{F}}_q} \left( T^{\overline{\mathbb{F}}_q} \right)$.
\end{theorem}
\begin{proof}
The equality $\GR_{\mathbb{F}_q}(T) = \GR_{\overline{\mathbb{F}}_q} \left( T^{\overline{\mathbb{F}}_q} \right)$ is derived from the definition of the geometric rank.  Let $\mathbb{F} = \mathbb{F}_{q^l}$ be a finite extension of $\mathbb{F}_q$.  By Lemma~\ref{lem3-1},  we have 
\begin{align*}
\AR_{\mathbb{F}} \left( T^{\mathbb{F}} \right) &=\sum_{i=1}^{d}n_i - n_d -\log_{\lvert \mathbb{F}\rvert}  \left( \left\lvert Z_d \left(  T^{\mathbb{F}}  \right) \right\rvert \right) \\
&= \sum_{i=1}^{d}n_i - n_d -\frac{\log_{\vert \mathbb{F}_{q}\vert}(\lvert Z_d( T\boxtimes_{\mathbb{F}_q} M_{d,\mathbb{F}} ) \rvert)}{l}.  
\end{align*}
Since $\mathbb{F}_q^{n_i} \otimes_{\mathbb{F}_{q}}  \mathbb{F} \simeq \mathbb{F}_q^{n_i l }$ as $\mathbb{F}_q$-vector spaces,  we may regard $T\boxtimes_{\mathbb{F}_q}  M_{d,\mathbb{F}}$ as a tensor in $\mathbb{F}_q^{n_1 l} \otimes_{\mathbb{F}_{q}} \cdots \otimes_{\mathbb{F}_{q}} \mathbb{F}_q^{n_d l}$.  By Lemma~\ref{lem3-1} again,  we have 
\[
\log_{\lvert \mathbb{F}_{q}\rvert}\left( \lvert Z_d( T\boxtimes_{\mathbb{F}_q} M_{d,\mathbb{F}} )\rvert \right) =  \sum_{i=1}^{d}n_i l - n_d l -\AR_{\mathbb{F}_{q}}( T\boxtimes_{\mathbb{F}_q}  M_{d,\mathbb{F}} ).
\]
Thus,  we obtain
\begin{equation}\label{equ1}
\AR_{\mathbb{F}} \left( T^{\mathbb{F}} \right) = \frac{\AR_{\mathbb{F}_{q}}( T\boxtimes_{\mathbb{F}_q} M_{d,\mathbb{F}} )}{l}.
    \end{equation}
Denote $s \coloneqq \Q_{\mathbb{F}_q} (M_{d,\mathbb{F}})$.  By definition, there exist matrices $L_{i} \in \mathbb{F}_{q}^{s \times l}$ for $1 \le i \le d$ such that $(L_{1}\otimes\cdots\otimes L_{d}) M_{d,\mathbb{F}} = \Id_s$,  where $\Id_s$ is the diagonal tensor in $\underbrace{\mathbb{F}_{q}^{s} \otimes_{\mathbb{F}_{q}} \cdots \otimes_{\mathbb{F}_{q}} \mathbb{F}_{q}^{s}}_{d \text{~copies}}$.  Therefore,  we have  
\[
\left( (I_{n_1} \boxtimes_{\mathbb{F}_q} L_1) \otimes \cdots \otimes (I_{n_d} \boxtimes_{\mathbb{F}_q} L_d) \right)  (T\boxtimes_{\mathbb{F}_q}  L_{d,\mathbb{F}}) = T \boxtimes_{\mathbb{F}_q} \Id_s = T^{\oplus s}.
\]
By Lemmas~\ref{lem3-3},  \ref{lem3-4} and Theorem~\ref{thm2-13}, we have 
\begin{equation}\label{equ2}
\AR_{\mathbb{F}_{q}}( T\boxtimes_{\mathbb{F}_{q}} M_{d,\mathbb{F}} )\ge \AR_{\mathbb{F}_{q}}(T^{\oplus s})= s\AR_{\mathbb{F}_{q}}(T)\ge c_{d,q}l\AR_{\mathbb{F}_{q}}(T)
\end{equation}
for some constant number $c(d,q) > 0$.  Combining \eqref{equ1} and \eqref{equ2},  we obtain 
\[
c_{d,q} \AR_{\mathbb{F}_{q}}(T)\le \AR_{\mathbb{F}}(T^{\mathbb{F}}).
\]

To conclude the proof,  we replace $\Q_{\mathbb{F}_q}(M_{d,\mathbb{F}})$ by $\R_{\mathbb{F}_q}(M_{d,\mathbb{F}})$ in the above argument and the other inequality follows immediately from Theorem~\ref{thm2-2}. 
\end{proof}
As an application of Theorem~\ref{thm3-5},  we will prove that the analytic rank and the geometric rank are the same up to a constant factor.  To this end,  we establish the following asymptotic relation between them with respect to field extensions.
\begin{theorem}[Geometric rank = limit of analytic rank]\label{thm3-7}
Given any tensor $T\in \mathbb{F}_q^{n_1} \otimes \cdots \otimes \mathbb{F}_q^{n_d}$,  we have 
\[
\lim_{ l \to\infty} \AR_{\mathbb{F}_{q^l}}\left( T^{\mathbb{F}_{q^l}} \right)= \GR_{\mathbb{F}_q}(T).
\]
\end{theorem}
\begin{proof}
Let $\mathbb{F} = \mathbb{F}_{q^l}$ be a field extension of $\mathbb{F}_q$.  We recall from \eqref{eq:Z} that $Z_d\left( T \right)$ is an affine cone in $\mathbb{F}_q^{m}$ where $m \coloneqq \sum_{i=1}^{d-1} n_i$.  Moreover,  we have $Z_d \left(T^{\mathbb{F}} \right) = Z_d\left( T \right)(\mathbb{F})$.  Let $\pi: \mathbb{F}^{m} \setminus \{0\} \to \mathbb{P}_{\mathbb{F}}^{m-1}$ be the projectivization map and let $W \coloneqq \pi \left( Z_d \left(  T^{\mathbb{F}} \right) \setminus \{0\} \right) $.  Clearly,  we have 
\begin{equation}\label{eq:dimWZ}
(q^l-1)\lvert  W \lvert + 1= \left\lvert  Z_d \left( T \right)(\mathbb{F})  \right\rvert.
\end{equation}
According to Theorem~\ref{thm3-6},  for each irreducible component $X$ of $W$,  there is a constant number $C(X) > 0$ such that 
\begin{equation}\label{eq:dimX}
\left\lvert \lvert X(\mathbb{F})\rvert - \frac{q^{l (r + 1)} - 1}{q^l-1}  \right\rvert \le (d-1)(d-2)q^{l(r-\frac{1}{2})}+ C_X q^{l(r -1)} = O \left( q^{l(r- \frac{1}{2})} \right),
\end{equation}
where $d \coloneqq \deg(X)$ and $r \coloneqq \dim (X)\le \dim Z_d(T) - 1 = m - \GR_{\mathbb{F}_q}(T) - 1$.  We notice that there exits a component $X$ such that $\dim (X) = m - \GR_{\mathbb{F}_q}(T) - 1$.  Hence by \eqref{eq:dimWZ} and \eqref{eq:dimX} we may derive
\[
q^{l(m- \GR_{\mathbb{F}_q}(T))}-O \left( q^{l \left( m- \GR_{\mathbb{F}_q}(T)- \frac{1}{2} \right)} \right) \le  \lvert  Z_d(T)(\mathbb{F}) \rvert \le C q^{l(m- \GR_{\mathbb{F}_q}(T))}
\]
for some constant number $C > 0$ that only depends on $T$ and $q$.  Therefore,  we have 
\[
\lim_{l\to \infty}\log_{q^l} \lvert  Z_d(T)(\mathbb{F}) \rvert = m- \GR_{\mathbb{F}_q}(T)
\]
and this together with Lemma~\ref{lem3-1} completes the proof.
\end{proof}
It is proved in \cite[Theorem~8.1]{KMZ23} that for any $T \in \mathbb{Z}^{n_1} \otimes \cdots \otimes \mathbb{Z}^{n_d}$,  
\begin{equation}\label{eq:ARp=GRZ}
\liminf_{p} \AR_{\mathbb{F}_p} \left( T^{\mathbb{F}_p} \right) = \GR_{\mathbb{Z}} (T),
\end{equation}
where $p$ runs through all prime numbers.  We remark that both Theorem~\ref{thm3-7} and \eqref{eq:ARp=GRZ} can be recognized as: the geometric rank of a tensor is equal to the limit of its analytic ranks under base changes.  Theorem~\ref{thm3-7} is concerned with this property for field extensions of a finite field,  while \eqref{eq:ARp=GRZ} validates this property for the base change from $\mathbb{Z}$ to $\mathbb{Z}_p$.  

As a direct consequence of Theorems~\ref{thm3-5} and \ref{thm3-7},  we have the proposition that follows.
\begin{proposition}[Geometric rank $\approx$ analytic rank]\label{coro3-8}
For any positive integer $d$ and prime power $q$,  there are constant numbers $C \coloneqq C (d,q) >0$ and $c \coloneqq c (d,q) >0$ such that 
\[
c \AR_{\mathbb{F}_{q}}(T)\le \GR_{\mathbb{F}_q}(T)\le C \AR_{\mathbb{F}_{q}}(T)
\]
 for any $T \in \mathbb{F}_{q}^{n_1}\otimes \cdots \otimes \mathbb{F}_{q}^{n_d}$.  
\end{proposition}
According to \cite[Proof of Theorem~1.13~(1)]{adiprasito2021schmidt},  we have $C \le (1- \log_q(d))^{-1}$.  Furthermore,  the proof of Theorem~\ref{thm3-5} indicates that $c \ge (96d^2q)^{-1}$.  As far as we are aware,  the existence of $c$ is not previously known.  Almost at the same time as this paper, Proposition~\ref{coro3-8} was independently proved in \cite[Theorem~1.2]{baily2024strength} by another method.

Let $\mathbb{F}$ be a field and let $\langle n,n,n \rangle_{\mathbb{F}}$ be the tensor in $\mathbb{F}^{n} \otimes \mathbb{F}^{n} \otimes \mathbb{F}^{n}$ corresponding to the $n \times n$ matrix multiplication.  By \cite[Theorem~3]{KMZ23},  we have $\GR_{\mathbb{F}} ( \langle n,n,n \rangle ) = \lceil 3n^2/4 \rceil$.  Applying Theorems~\ref{thm3-5},  \ref{thm3-7} and Proposition~\ref{coro3-8} to $\langle n,n,n \rangle_{\mathbb{F}_q}$,  we obtain the following result,  which may be of independent interest.
\begin{corollary}[Analytic rank of matrix multiplication]\label{cor:AR}
For any prime power $q$,  we have:
\begin{enumerate}[(a)]
\item For any integer $l > 0$,
\[
(432q)^{-1} \AR_{\mathbb{F}_q} \left( \langle n,n,n \rangle_{\mathbb{F}_q}\right) \le \AR_{\mathbb{F}_{q^l}} \left( \langle n,n,n \rangle_{\mathbb{F}_{q^l}} \right) \le (1 - \log_q(3))^{-1} \AR_{\mathbb{F}_q}(\langle n,n,n \rangle_{\mathbb{F}_q}).
\] \label{cor:AR:item1}
\item For each $\varepsilon > 0$,  there exists an integer $l_0 > 0$ such that for any $l \ge l_0$
\[
\left\lvert \AR_{\mathbb{F}_{q^l}}\left(  \langle n,n,n \rangle_{\mathbb{F}_{q^l}}\right) - \lceil 3n^2/4 \rceil \right\rvert \le \varepsilon.
\] \label{cor:AR:item2}
\item $(1 - \log_q(3)) \lceil 3n^2/4 \rceil  \le \AR_{\mathbb{F}_q}( \langle n,n,n \rangle_{\mathbb{F}_q} ) \le 108q(3n^2 + 1)$.  In particular,  
\[
\limsup_{n\to \infty} \log_n \AR_{\mathbb{F}_q}( \langle n,n,n \rangle_{\mathbb{F}_q} ) = 2.
\] \label{cor:AR:item3}
\end{enumerate}
\end{corollary}
We remark that \eqref{cor:AR:item2} and \eqref{cor:AR:item3} can also be obtained by a direct calculation.  Indeed,  it is straightforward to verify that $\AR_{\mathbb{F}_q}( \langle n,n,n \rangle_{\mathbb{F}_q} ) = \log_q \left\lvert \lbrace (A,B) \in \mathbb{F}_q^{n\times n} \times \mathbb{F}_q^{n\times n}: AB = 0 \rbrace \right\rvert$,  from which $\AR_{\mathbb{F}_q}( \langle n,n,n \rangle_{\mathbb{F}_q} )$ can be estimated by counting the number of $B \in \mathbb{F}_q^{n\times n}$ such that $AB = 0$ when $\operatorname{rank}(A) = r$ for each $0 \le r \le n$.
\section{Partition rank and slice rank under field extension}
In this section,  we investigate the behaviour of the partition rank and the slice rank under field extensions. Let $\mathbb{K}$ be a field and let $\mathbb{F}$ be a finite extension of $\mathbb{K}$.  We recall that any $\mathbb{F}$-tensor $S \in \mathbb{F}^{n_1} \otimes_{\mathbb{F}} \cdots \otimes_{\mathbb{F}} \mathbb{F}^{n_d}$ restricts to a $\mathbb{K}$-tensor $S_{\mathbb{K}} \in \mathbb{F}^{n_1} \otimes_{\mathbb{K}} \cdots \otimes_{\mathbb{K}} \mathbb{F}^{n_d}$ defined by a map similar to that in \eqref{eq:iota}.  Thus,  both $\pr_{\mathbb{F}}(S)$ and $\pr_{\mathbb{K}}(S)\coloneqq \pr_{\mathbb{K}} (S_{\mathbb{K}})$ are well-defined. Moreover,  they are related by an inequality.  
\begin{lemma}\label{lemPR=1}   
Suppose that $\mathbb{K}$ is either a finite field or a perfect infinite field and $\mathbb{F}$ is a finite extension of $\mathbb{K}$.  Given any $S\in \mathbb{F}^{n_1} \otimes_{\mathbb{F}} \cdots \otimes_{\mathbb{F}} \mathbb{F}^{n_d}$,  we have 
\[
\pr_{\mathbb{F}}(S)  \le \pr_{\mathbb{K}}(S_{\mathbb{K}}) \le  \R_{\mathbb{K}}(M_{3,\mathbb{F}}) \pr_{\mathbb{F}}(S).
\]
The same is true if we replace the partition rank in the above statement by slice rank or cp-rank,  with the only modification that $\R_{\mathbb{K}}(M_{3,\mathbb{F}})$ is substituted by $\R_{\mathbb{K}}(M_{d+1,\mathbb{F}})$ for cp-rank. 
\end{lemma} 
\begin{proof}
We first establish $\pr_{\mathbb{F}}(S)  \le \pr_{\mathbb{K}}(S_{\mathbb{K}})$.  Since $\mathbb{F}$ is a separable extension of $\mathbb{K}$,  we may write $\mathbb{F} = \mathbb{K}(\alpha)$ for some $\alpha \in \mathbb{F}$.  For each $1 \le k \le d$,  we let $\{ x^{(k)}_{j}:  1 \le j \le n_k \}$ be an $\mathbb{F}$-basis of $\mathbb{F}^{n_k}$.   Then $\{ \alpha^s x^{(k)}_j: 0 \le s \le l-1, \; 1 \le j \le n_k \}$ is a $\mathbb{K}$-basis of $\mathbb{F}^{n_k}$,  where $l \coloneqq [\mathbb{F}: \mathbb{K}]$.  Thus,  we may write 
\[
S_{\mathbb{K}} = \sum_{\substack{j_1,\dots,  j_d\\ s_1,\dots,  s_d}} S^{j_1,\dots,  j_d}_{s_1,\dots,  s_d} \left( \alpha^{s_1} x^{(1)}_{j_1} \right) \otimes_{\mathbb{K}} \cdots \otimes_{\mathbb{K}} \left( \alpha^{s_d} x^{(d)}_{j_d} \right) \in \mathbb{F}^{n_1} \otimes_{\mathbb{K}} \cdots \otimes_{\mathbb{K}} \mathbb{F}^{n_d}.
\]
where $S^{j_1,\dots,  j_d}_{s_1,\dots,  s_d} \in \mathbb{K}$,  $1 \le j_k \le n_k$ and $0 \le s_k \le l-1$ for each $1 \le k \le d$.  By definition of $S_{\mathbb{K}}$,  we have
\[
S( (x^{(1)}_{j_1})^\ast,\dots,  (x^{(d-1)}_{j_{d-1}})^\ast ) = S_{\mathbb{K}}(x^{(1)}_{j_1},\dots,  x^{(d-1)}_{j_{d-1}}) = \sum_{s=0,j_d=1}^{l-1,n_d} S^{j_1,\dots,  j_d}_{0,\dots,  0,  s} \alpha^{s}  x^{(d)}_{j_d}.
\] 
Here $\{ (x^{(k)}_{j})^\ast: 1 \le j \le n_k  \}$ is the $\mathbb{F}$-basis of $(\mathbb{F}^{n_k})^\ast$ dual to $\{  x^{(k)}_{j}: 1 \le j \le n_k \}$ for $1 \le k \le d-1$.  Suppose $\{  y^{(k)}_{s_k,j_k}: 0\le s_k \le l-1,\; 1 \le j \le n_k \}$ is an $\mathbb{F}$-basis of $\mathbb{F}^{ln_k}$ with the dual basis $\{  (y^{(k)}_{s_k,j_k})^\ast: 0\le s_k \le l-1,\; 1 \le j \le n_k \}$.  We consider an $\mathbb{F}$-tensor:
\[
(S_{\mathbb{K}})^{\mathbb{F}} = \sum_{\substack{j_1,\dots, j_d \\ s_1,\dots, s_d}} S^{j_1,\dots,  j_d}_{s_1,\dots,  s_d} y^{(1)}_{s_1,j_1} \otimes_{\mathbb{F}} \cdots \otimes_{\mathbb{F}} y^{(d)}_{s_d,j_d} \in \mathbb{F}^{l n_1} \otimes_{\mathbb{F}} \cdots \otimes_{\mathbb{F}} \mathbb{F}^{l n_d}.
\]
  
For $1 \le k \le d-1$,  we let $\tau_k: \left( \mathbb{F}^{n_k} \right)^\ast \to \left( \mathbb{F}^{l n_k} \right)^\ast$ be the $\mathbb{F}$-linear map obtained by linearly extending the assignment $( x^{(k)}_{j})^\ast \mapsto ( y^{(k)}_{0,j} )^\ast$,  $1 \le j \le n_k$.   We also define $p: \mathbb{F}^{l n_d} \to \mathbb{F}^{n_d}$ to be the $\mathbb{F}$-linear map obtained by linearly extending the assignment $y^{(d)}_{s,j}  \mapsto \alpha^{s}  x^{(d)}_{j}$,  $1 \le j \le n_d$,  $0 \le s \le l-1$.  We claim the the following digram commutes:
\[\begin{tikzcd}
	{(\mathbb{F}^{n_1})^\ast \times \cdots \times (\mathbb{F}^{n_{d-1}})^\ast} && {(\mathbb{F}^{ln_1})^\ast \times \cdots \times (\mathbb{F}^{ln_{d-1}})^\ast} \\
	{\mathbb{F}^{n_d}} && {\mathbb{F}^{ln_d}}
	\arrow["{\tau_1 \times \cdots \times \tau_{d-1}}", from=1-1, to=1-3]
	\arrow["S"', from=1-1, to=2-1]
	\arrow["{{S^{\mathbb{F}}}}", from=1-3, to=2-3]
	\arrow["p", from=2-3, to=2-1]
\end{tikzcd}\]
Indeed,  it is straightforward to verify that 
\begin{align*}
p\left( (S_{\mathbb{K}})^{\mathbb{F}} \left( \tau_1( (x^{(1)}_{j_1})^\ast ),\dots,  \tau_{d-1}( (x^{(d-1)}_{j_{d-1}})^\ast ) \right) \right) 
&= p \left( (S_{\mathbb{K}})^{\mathbb{F}} \left( (y^{(1)}_{0,j_1})^\ast,  \dots,  (y^{(d-1)}_{0,j_{d-1}})^\ast \right) \right) \\
&=p\left(  \sum_{s=0,j_d=1}^{l-1,n_d} S^{j_1,\dots, j_d}_{0,\dots,  0,  s} y^{(d)}_{s,j_d}  \right) \\
&= \sum_{s=0,j_d=1}^{l-1,n_d} S^{j_1,\dots, j_d}_{0,\dots,  0,  s} \alpha^s x^{(d)}_{j_d} \\
&= S \left( (x^{(1)}_{j_1})^\ast,\dots,  (x^{(d-1)}_{j_{d-1}})^\ast \right).
\end{align*}
Thus,  we may conclude that $ \pr_{\mathbb{K}} (S_{\mathbb{K}}) \ge  \pr_{\mathbb{F}} ( (S_{\mathbb{K}})^{\mathbb{F}}) \ge \pr_{\mathbb{F}} ( S )$.

Next,  we prove $\pr_{\mathbb{K}}(S_{\mathbb{K}}) \le  \R_{\mathbb{K}}(M_{3,\mathbb{F}}) \pr_{\mathbb{F}}(S)$.  Denote $r \coloneqq \pr_{\mathbb{F}}(S)$.  Then we may write $S = \sum_{i = 1}^r S_i$ for some $S_i \in \mathbb{F}^{n_1} \otimes_{\mathbb{F}} \cdots \otimes_{\mathbb{F}} \mathbb{F}^{n_d}$ such that $\pr_{\mathbb{F}}(S_{i})= 1$.  This implies that $\pr_{\mathbb{K}}(S) \le \sum_{i=1}^r \pr_{\mathbb{K}}(S_i)$ and it reduces to prove the inequality for $r = 1$. If $\pr_{\mathbb{F}}(S)=1$,  we may write $S = a \otimes b$ for some $a\in \mathbb{F}^{n_{\pi(1)}} \otimes_{\mathbb{F}} \cdots \otimes_{\mathbb{F}} \mathbb{F}^{n_{\pi(p)}}$ and $b \in \mathbb{F}^{n_{\pi(p+1)}} \otimes_{\mathbb{F}} \cdots \otimes_{\mathbb{F}} \mathbb{F}^{n_{\pi(d)}}$.  Here $\pi \in \mathfrak{S}_d$ and $a \otimes b$ is identified with the tensor in $\mathbb{F}^{n_1} \otimes_{\mathbb{F}} \cdots \otimes_{\mathbb{F}} \mathbb{F}^{n_d}$ via the natural isomorphism 
\[
\mathbb{F}^{n_{\pi(1)}} \otimes_{\mathbb{F}} \cdots \otimes_{\mathbb{F}} \mathbb{F}^{n_{\pi(d)}} \simeq \mathbb{F}^{n_1} \otimes_{\mathbb{F}} \cdots \otimes_{\mathbb{F}} \mathbb{F}^{n_d}.
\]
For convenience,  we always adopt this convention in the sequel.  We observe that $\mathbb{F}^{n_{\pi(1)}} \otimes_{\mathbb{F}} \cdots \otimes_{\mathbb{F}} \mathbb{F}^{n_{\pi(p)}} = \left( \mathbb{K}^{n_{\pi(1)}} \otimes_{\mathbb{K}} \cdots \otimes_{\mathbb{K}}  \mathbb{K}^{n_{\pi(p)}} \right) \otimes_{\mathbb{K}} \mathbb{F}$,  from which we may identify $a$ with a $\mathbb{K}$-multilinear map $f_a: \mathbb{K}^{n_{\pi(1)}} \times \cdots \times \mathbb{K}^{n_{\pi(p)}} \to \mathbb{F}$.  Similarly,  we identify $b$ (resp.  $a\otimes b$) with a $\mathbb{K}$-multilinear map $f_b: \mathbb{K}^{n_{\pi(p+1)}} \times \cdots \times \mathbb{K}^{n_{\pi(d)}} \to \mathbb{F}$ (resp.  $f_{a \otimes b}: \mathbb{K}^{n_1} \times \cdots \times \mathbb{K}^{n_d} \to \mathbb{F}$).  We notice that as $\mathbb{K}$-multilinear maps,  it holds that $f_{a\otimes b} = M_{3,\mathbb{F}} \circ (f_a \otimes f_b)$.  Suppose that $M_{3,\mathbb{F}} = \sum_{j=1}^s \alpha_j \otimes \beta_j \otimes u_j$ where $s = R_{\mathbb{K}}(M_{3,\mathbb{F}})$,  $\alpha_j$ and $\beta_j$ are $\mathbb{K}$-linear forms for $1 \le j \le s$.  Then we can write 
\[
f_{a\otimes b} = \sum_{j=1}^s (\alpha_j \circ f_a) \left[ (\beta_j \circ f_b) u_j \right].
\]
This implies that $\pr_{\mathbb{K}} (a \otimes b) \le s = R_{\mathbb{K}}(M_{3,\mathbb{F}})$.

To prove the two inequalities for slice rank and cp-rank,  one can simply repeat the above argument while replacing partition rank and the corresponding decomposition with those of slice rank and cp-rank,  respectively.
\end{proof} 
\begin{remark}
It is proved in \cite[Proof of Corollary~1]{cohen2023partition} that
\begin{equation}\label{eq:PR1}
\operatorname{PR}_{\mathbb{K}} (T) \le l \operatorname{PR}_{\mathbb{F}}(T^{\mathbb{F}}),
\end{equation}
where $\mathbb{K}$ is a finite field,  $\mathbb{F}$ is a degree $l$ extension of $\mathbb{K}$ and $T$ is a $\mathbb{K}$-tensor.  It is worth noticing that \eqref{eq:PR1} is different from the one established in Lemma~\ref{lemPR=1},  although they appear similar.  The main difference is that \eqref{eq:PR1} is concerned with the relation between partition ranks of a $\mathbb{K}$-tensor and its $\mathbb{F}$-extension,  whereas Lemma~\ref{lemPR=1} relates the partition rank of an $\mathbb{F}$-tensor and that of its $\mathbb{K}$-restriction.  

On the other hand,  one of the anonymous reviewers has pointed out that the inequality $\operatorname{PR}_{\mathbb{K}} (S_{\mathbb{K}}) \le l \operatorname{PR}_{\mathbb{F}} (S)$ holds.  Indeed,  according to Remark~\ref{rmk:digression},  the $\mathbb{K}$-multilinear function corresponding to $S_{\mathbb{K}}$ is actually ${\tr} \circ f_{a\otimes b}$,  where we assume $S = a \otimes b$ as in the proof of Lemma~\ref{lemPR=1}.  Since $\operatorname{rank} \left( {\tr} \circ M_{3,\mathbb{F}} \right) = l$ and 
\[
{\tr} \circ f_{a\otimes b} = \left( {\tr} \circ M_{3,\mathbb{F}} \right) \circ (f_a \otimes f_b),
\]
the desired inequality follows immediately.  The same argument applies to slice rank and cp-rank.  Consequently,  we may improve the constant factor from $\operatorname{R_{\mathbb{K}}}(M_{d+1,  \mathbb{F}})$ to $\operatorname{R_{\mathbb{K}}}(M_{d,  \mathbb{F}})$. 
\end{remark}

\begin{lemma}\label{lem4-1}
Suppose $\mathbb{K}$ is either a finite field or a perfect infinite field.  Let $T\in \mathbb{K}^{n_1} \otimes \cdots \otimes \mathbb{K}^{n_d}$ be a tensor.  If there exists a constant $c \coloneqq c (d,\mathbb{K}) > 0$ such that 
\[
\limsup_{n\to\infty} \frac{\pr_{\mathbb{K}} \left( T^{\oplus n} \right)}{n}\ge c \pr_{\mathbb{K}} (T),
\] 
then there exists a constant $C \coloneqq C (d,\mathbb{K}) > 0$ such that 
\[
\pr_{\mathbb{K}}(T)\le C \pr_{\overline{\mathbb{K}}}\left( T^{\overline{\mathbb{K}}} \right).
\]
The same is true if we replace the partition rank in the above statement by slice rank or cp-rank.
\end{lemma}  
\begin{proof} 
Proofs for slice rank and cp-rank are the same as that for partition rank,  thus we only prove for partition rank.  Denote $r \coloneqq \pr_{\overline{\mathbb{K}}}\left( T^{\overline{\mathbb{K}}} \right)$.  Then for each $1 \le j \le r$,  there are $\pi \coloneqq \pi_j \in \mathfrak{S}_d$,  $1 \le p \coloneqq p_j \le d$,  $a_{j} \in \overline{\mathbb{K}}^{n_{\pi(1)}} \otimes_{\overline{\mathbb{K}}}  \cdots \otimes_{\overline{\mathbb{K}}}  \overline{\mathbb{K}}^{n_{\pi(p)}}$ and $b_{j} \in \overline{\mathbb{K}}^{n_{\pi(p+1)}} \otimes_{\overline{\mathbb{K}}}  \cdots \otimes_{\overline{\mathbb{K}}}  \overline{\mathbb{K}}^{n_{\pi(d)}}$ such that 
\begin{equation}\label{lem4-1:deompT}
T^{\overline{\mathbb{K}}} = \sum_{j=1}^r a_j \otimes b_j.
\end{equation}
Since $a_j$'s and $b_j$'s are tensors over $\overline{\mathbb{K}}$,  there exists a finite extension $\mathbb{F}/\mathbb{K}$ such that for each $1 \le j \le r$,
\begin{equation}\label{lem4-1:reduction}
a_{j} \in \mathbb{F}^{n_{\pi(1)}} \otimes_{\mathbb{F}} \cdots \otimes_{\mathbb{F}} \mathbb{F}^{n_{\pi(p)}},\quad b_{j} \in \mathbb{F}^{n_{\pi(p+1)}} \otimes_{\mathbb{F}} \cdots \otimes_{\mathbb{F}} \mathbb{F}^{n_{\pi(d)}}.
\end{equation}
In particular,  we have  
\[
 r = \pr_{\overline{\mathbb{K}}}\left( T^{\overline{\mathbb{K}}} \right) = \pr_{\mathbb{F}}\left( T^{\mathbb{F}} \right),\quad  (T^{\mathbb{F}})_{\mathbb{K}} = T \boxtimes_{\mathbb{K}} M_{d,\mathbb{F}}.
\]

By Lemma~\ref{lemPR=1}, Proposition~\ref{lem2-1} and Theorem~\ref{thm2-2}, there exists $c_1 \coloneqq c_1 (d,\mathbb{K})$ such that $\pr_{\mathbb{K}}(T\boxtimes_{\mathbb{K}} M_{d,\mathbb{F}})\le c_1 l r$.  On the other side,  Proposition~\ref{lem2-1} and Theorem~\ref{thm2-13} implies that there exists some constant $c_2 \coloneqq c_2 (d,\mathbb{K})$ such that $\Q_{\mathbb{K}}(M_{d,\mathbb{F}}) \ge c_2 l$.  This implies $T \boxtimes_{\mathbb{K}} M_{d,\mathbb{F}} \succeq T^{\oplus c_2 l}$ and by the assumption we have   
\[
c_1 l r  \ge \pr_{\mathbb{K}} ( T \boxtimes_{\mathbb{K}} M_{d,\mathbb{F}} ) \ge  c c_2   l \pr_{\mathbb{K}}(T).
\]
Therefore,  we obtain the desired inequality by setting $C \coloneqq c _1 / c c_2$.   
\end{proof}
We notice that the assumption in Lemma~\ref{lem4-1} is satisfied by slice rank.  Indeed,  according to Lemma~\ref{lem4-3},  we may take $c = 1$ for slice rank.  As an application of Lemma~\ref{lem4-1},  we derive the stability of the slice rank under field extensions.
\begin{proposition}[Stability of slice rank]\label{coro4-4}
Let $T\in \mathbb{K}^{n_1} \otimes_{\mathbb{K}} \cdots  \otimes_{\mathbb{K}} \mathbb{K}^{n_d}$ be a tensor.  If $\mathbb{K}$ is a perfect infinite field,  then we have 
\[
\sr_{\mathbb{K}}(T) \le d^2 \sr_{\overline{\mathbb{K}}}\left( T^{\overline{\mathbb{K}}} \right).
\]
If $\mathbb{K} = \mathbb{F}_q$,  then there is a constant number $C  \coloneqq C(q,d)$ such that
\[
\sr_{\mathbb{F}_q}(T) \le C \sr_{\overline{ \mathbb{F}}_q }\left( T^{\overline{ \mathbb{F}}_q } \right).
\]
\end{proposition}
\begin{proof}
If $\mathbb{K}$ is a perfect infinite field.  In the proof of Lemma~\ref{lem4-1},  we may choose $c_1 = d$ and $c_2 = 1/d$.  Thus we have $C = d^2$ in Lemma~\ref{lem4-1}.  If $q$ is a power of $p$,  then we have
\[
\sr_{\mathbb{K}}(T) \le \sr_{\mathbb{F}_p}(T) \le  C(p, d) \sr_{\overline{\mathbb{F}}_p}\left( T^{\overline{\mathbb{F}}_p} \right) = C(p,  d) \sr_{\overline{\mathbb{K}}}\left( T^{\overline{\mathbb{K}}} \right),
\]
where the first inequality follows from the definition,  the second inequality is obtained by Lemma~\ref{lem4-1} and the equality is the consequence of the fact that $\overline{\mathbb{K}} = \overline{\mathbb{F}}_p$.
\end{proof}
\begin{remark}\label{rem4-5}
It is proved in \cite{derksen2022g} that
\[
\frac{2 \sr_{\mathbb{K}}(T)}{d} \le \R^G_{\mathbb{K}}(T) = \R^G_{\overline{\mathbb{K}}}(T) \le \sr_{\overline{\mathbb{K}}}(T),
\]
where $\R^G_{\mathbb{K}}(T)$ is the G-stable rank of $T$.  This leads to the stability of the slice rank under field extensions with a better constant:
\begin{equation}\label{rem4-5:eq1}
\sr_{\mathbb{K}}(T)\le \frac{d}{2} \sr_{\overline{\mathbb{K}}}\left( T^{\overline{\mathbb{K}}} \right).
\end{equation}
\end{remark}

We recall that for any positive integers $t \le s,  m_1,\dots,  m_s$,  there is a natural injective map $\tau$ defined by
\begin{equation}\label{eq:injection}
\begin{tikzcd}
	{\mathbb{K}^{m_1} \otimes_{\mathbb{K}} \cdots  \otimes_{\mathbb{K}} \mathbb{K}^{m_t}} & {\mathbb{K}\left[(\mathbb{K}^{m_1})^{\ast} \oplus \cdots \oplus (\mathbb{K}^{m_s})^{\ast} \right]} \\
	{\Hom\left( (\mathbb{K}^{m_1})^{\ast}  \otimes_{\mathbb{K}} \cdots  \otimes_{\mathbb{K}} (\mathbb{K}^{m_t})^{\ast},\mathbb{K}\right)} & {\mathbb{K}\left[(\mathbb{K}^{m_1})^{\ast} \oplus \cdots \oplus (\mathbb{K}^{m_t})^{\ast} \right] }
	\arrow["\tau", hook, from=1-1, to=1-2]
	\arrow["\simeq"', from=1-1, to=2-1]
	\arrow[hook, from=2-1, to=2-2]
	\arrow[hook, from=2-2, to=1-2]
\end{tikzcd}
\end{equation}
where $\mathbb{K}[\mathsf{V}]$ denotes the ring of polynomials on a vector space $\mathsf{V}$.  The image of this injection consists of all multilinear polynomials on $(\mathbb{K}^{m_1})^{\ast} \oplus \cdots \oplus (\mathbb{K}^{m_s})^{\ast}$.  The following inequality is stated in \cite[Theorem~5]{KMZ23} with a sketched proof involving several lemmas.  We provide below a shorter proof.
\begin{lemma}
Let $\mathbb{K}$ be a field and let $T\in \mathbb{K}^{n_1}  \otimes_{\mathbb{K}} \cdots  \otimes_{\mathbb{K}} \mathbb{K}^{n_d}$ be a tensor.  Then $\GR_{\mathbb{K}}(T)\le \pr_{\mathbb{K}}(T)$.
\label{lem-PRGR}
\end{lemma}
\begin{proof}
Since $\GR_{\mathbb{K}}(T) = \GR_{\overline{\mathbb{K}}}\left(T^{\overline{\mathbb{K}}}\right)$ and $\pr_{\overline{\mathbb{K}}} \left(T^{\overline{\mathbb{K}}}\right) \le \pr_{\mathbb{K}}(T)$,  we may assume that $\mathbb{K}$ is algebraically closed.  We define an $n_d$-dimensional subspace $\mathsf{W} \coloneqq \langle T,  (\mathbb{K}^{n_{d}})^{\ast} \rangle$ in $\mathbb{K}^{n_1} \otimes_{\mathbb{K}} \cdots  \otimes_{\mathbb{K}} \mathbb{K}^{n_{d-1}}$.  Denote $\mathsf{V} \coloneqq (\mathbb{K}^{n_{1}})^{\ast}\oplus\cdots \oplus (\mathbb{K}^{n_{d-1}})^{\ast}$.  Let $\tau$ be the injection defined in \eqref{eq:injection} and let $J$ be the ideal generated by $\tau(\mathsf{W})$ in $\mathbb{K}[\mathsf{V}]$.  Obviously,  the variety defined by $J$ coincides with $Z_d(T)$ defined in \eqref{eq:GR} and $\GR_{\mathbb{K}}(T) = \codim J$.

Suppose that $T= \sum_{i=1}^{r} a_{i}\otimes b_{i}$ is a partition rank decomposition of $T$ where $r \coloneqq \pr_{\mathbb{K}}(T)$ and $a_i$ dose not involve the $d$-th factor for each $1 \le i \le r$.  Thus,  we have $r$ polynomials: 
\[
f_i \coloneqq \tau (a_i) \in \mathbb{K}[\mathsf{V}],  1 \le i \le r.
\]
Since $\mathsf{W} = \langle T,  (\mathbb{K}^{n_d})^\ast \rangle = \sum_{i=1}^{r} a_{i}\otimes \langle b_{i},   (\mathbb{K}^{n_d})^\ast \rangle$ and $J$ is generated by $\mathsf{W}$,  we may derive $J \subseteq (f_1,\dots,  f_r)$.  According to \cite[Theorem~10.2]{Eisenbud95},  we conclude that $\codim J \le r$.
\end{proof}
In the following,  we establish the equivalence among four conjectures concerning the partition rank.
\begin{theorem}[Equivalence of conjectures] \label{thm4-6}
Let $d$ be a positive integer and  let $q$ be a power of $p$.  The following are equivalent:
    \begin{enumerate}[(a)]
        \item Partition rank v.s.  analytic rank conjecture \cite{GW11,Milicevic19, adiprasito2021schmidt}: there exists a constant $C_1 \coloneqq C_1(d,q) > 0$ such that 
\[
\pr_{\mathbb{F}_q}(T)\le C_1 \AR_{\mathbb{F}_{q}}(T)
\] 
for any $T \in \mathbb{F}_q^{n_1} \otimes_{\mathbb{F}_q} \cdots \otimes_{\mathbb{F}_q} \mathbb{F}_q^{n_d}$. \label{thm4-6-1}
\item Partition rank v.s.  geometric rank conjecture \cite{Schmidt85, adiprasito2021schmidt, cohen2023partition}: there exists a constant $C_2 \coloneqq C_{2} (d,q) >0$ such that 
\[
\pr_{\mathbb{F}_{q}}(T)\le C_{2} \GR_{\mathbb{F}_{q}}(T)
\]        
for any $T \in \mathbb{F}_q^{n_1} \otimes_{\mathbb{F}_q} \cdots \otimes_{\mathbb{F}_q} \mathbb{F}_q^{n_d}$.\label{thm4-6-2}
        \item Asymptotic direct sum conjecture for partition rank: there exists a constant $C_3 \coloneqq C_3(d,q) > 0$ such that 
\[
\pr_{\mathbb{F}_q} (T) \le C_3 \limsup_{n\to\infty} \frac{\pr_{\mathbb{F}_{q}}\left( T^{\oplus n} \right)}{n} 
\]        
for any $T \in \mathbb{F}_q^{n_1} \otimes_{\mathbb{F}_q} \cdots \otimes_{\mathbb{F}_q} \mathbb{F}_q^{n_d}$.\label{thm4-6-3}
        \item Stability conjecture for partition rank \cite{adiprasito2021schmidt}: there exists a constant $C_4 \coloneqq C_4(d,q) > 0$ such that 
\[
\pr_{\mathbb{F}_{q}}(T)\le C_4  \pr_{\overline{\mathbb{F}}_{q}} \left( T^{\overline{\mathbb{F}}_{q}} \right)
\]
for any $T \in \mathbb{F}_q^{n_1} \otimes_{\mathbb{F}_q} \cdots \otimes_{\mathbb{F}_q} \mathbb{F}_q^{n_d}$.\label{thm4-6-4} 
\end{enumerate}
Moreover,  \eqref{thm4-6-1}--\eqref{thm4-6-4} hold for any powers of $p$ if and only if they hold for $p$.
\end{theorem}
\begin{proof}
The implication \eqref{thm4-6-1} $\implies $ \eqref{thm4-6-2} follows from Proposition~\ref{coro3-8}.  We notice that by definition,  
\[
\GR_{\mathbb{F}_q} (T\oplus S)= \GR_{\mathbb{F}_q}(T)+ \GR_{\mathbb{F}_q}(S).
\]
Thus,  we have $\pr_{\mathbb{F}_q}\left( T^{\oplus n} \right) \ge \GR_{\mathbb{F}_q}\left( T^{\oplus n}\right) = n \GR_{\mathbb{F}_q}(T)$ by Lemma~\ref{lem-PRGR}.  This proves \eqref{thm4-6-2} $\implies$ \eqref{thm4-6-3}.  If \eqref{thm4-6-3} holds,  then \eqref{thm4-6-4} also holds by Lemma~\ref{lem4-1}.  The implication \eqref{thm4-6-4} $\implies$ \eqref{thm4-6-1} is obtained by combining \cite[Corollary~3]{cohen2023partition} and the proof of \cite[Theorem~1.13]{adiprasito2021schmidt}.

One direction of the `moreover' part is trivial.  To prove the other direction,  it is sufficient to prove that \eqref{thm4-6-3} holds for $q = p^l$ for any $l\in \mathbb{N}$ as long as it holds for $p$.  Given a tensor $T\in \mathbb{F}_{q}^{n_{1}} \otimes_{\mathbb{F}_q} \cdots \otimes_{\mathbb{F}_q} \mathbb{F}_{q}^{n_{d}}$,  Theorem~\ref{thm2-2} and Lemma~\ref{lemPR=1} indicate the existence of a constant $C(d,p)$ such that $\pr_{\mathbb{F}_{p}}(T^{\oplus n})\le C(d,p)l\pr_{\mathbb{F}_{q}}(T^{\oplus n})$.  Since \eqref{thm4-6-3} holds for $p$,  we also have $c(d,p)n\pr_{\mathbb{F}_{p}}(T)\le\pr_{\mathbb{F}_{p}}(T^{\oplus n})$.  We obtain 
\[
\dfrac{c(d,p) \pr_{\mathbb{F}_{q}}(T)}{C(d,p)l}\le \dfrac{c(d,p) \pr_{\mathbb{F}_{p}}(T)}{C(d,p)l}\le\dfrac{ \pr_{\mathbb{F}_{p}}(T^{\oplus n}) }{C(d,p)ln} \le \dfrac{\pr_{\mathbb{F}_{q}}(T^{\oplus n})}{n},  
\]
where the first inequality follows from Lemma~\ref{lemPR=1}.  Hence \eqref{thm4-6-3} for $q$ is verified by taking $C'(d,q) \coloneqq C(d,p)l/c(d,p)$.
 \end{proof}
We remark that the dependence of constants $C_1,\dots,  C_4$ on $q$ in Theorem~\ref{thm4-6} can be removed by a similar argument \cite{GD24}.  However,  Theorem~\ref{thm4-6} and its base field independence version are actually parallel,  in the sense that one can not be obtained from the other. 

Next we consider the slice rank of a linear subspace of tensors.  Given a linear subspace $\mathsf{W} \subseteq \mathbb{K}^{n_1} \otimes_{\mathbb{K}}  \cdots \otimes_{\mathbb{K}}   \mathbb{K}^{n_d}$ over a field $\mathbb{K}$,  the \emph{slice rank of $\mathsf{W}$} is defined as 
\begin{equation}\label{eq:SRW}
\sr_{\mathbb{K}}(\mathsf{W}) \coloneqq \min \left\lbrace \sum_{i=1}^{d} \codim( \mathsf{U}_{i}):  \langle \mathsf{W}, \mathsf{U}_{1} \otimes_{\mathbb{K}}  \cdots \otimes_{\mathbb{K}}  \mathsf{U}_{d}\rangle=0,\; \mathsf{U}_i \subseteq (\mathbb{K}^{n_i})^{\ast},\;1 \le i  \le d \right\rbrace.
\end{equation}
Here $\langle \mathsf{W}, \mathsf{U}_{1}\otimes \cdots\otimes  \mathsf{U}_{d}\rangle = 0$ means that $\langle T,  L \rangle = 0$ for each pair $(T, L) \in \mathsf{W} \times \left( \mathsf{U}_{1}\otimes \cdots\otimes \mathsf{U}_{d} \right)$.  In the sequel,  we denote 
\begin{equation}\label{eq:Vi}
\mathsf{V}_i \coloneqq \mathbb{K}^{n_1} \otimes_{\mathbb{K}}  \cdots \otimes_{\mathbb{K}}  \mathbb{K}^{n_{i-1}} \otimes_{\mathbb{K}}  \mathbb{K}^{n_{i+1}} \otimes_{\mathbb{K}}  \cdots \otimes_{\mathbb{K}}  \mathbb{K}^{n_d},
\end{equation}
for each $1 \le i \le d$.
\begin{lemma}[Characterization of $\sr_{\mathbb{K}}(\mathsf{W})$]\label{lem4-8}
Let $\mathbb{K}$ be a field.  For a linear subspace $\mathsf{W} \subseteq \mathbb{K}^{n_1} \otimes_{\mathbb{K}}  \cdots \otimes_{\mathbb{K}}  \mathbb{K}^{n_d}$,  $\sr_{\mathbb{K}}(\mathsf{W})$ is the minimal positive integer $r$ to ensure the existence of $r_i \in \mathbb{N}$ and $a_{i1},\dots,  a_{ir_i} \in \mathbb{K}^{n_i}$ for $1 \le i \le d$,  such that $r = \sum_{i=1}^d r_i$ and every $T \in \mathsf{W}$ admits a decomposition
\begin{equation}\label{lem4-8:eq1}
T = \sum_{i=1}^d  \sum_{j=1}^{r_i} a_{ij} \otimes T_{ij}
\end{equation}
for some $T_{i1},\dots,  T_{ir_i} \in \mathsf{V}_i$,  $1 \le i \le d$.  
\end{lemma}
\begin{proof}
We denote by $s$ the minimal integer such that \eqref{lem4-8:eq1} holds.  We prove that $s = t \coloneqq \sr_{\mathbb{K}}(\mathsf{W})$.  If $\langle \mathsf{W}, \mathsf{U}_{1} \otimes_{\mathbb{K}}  \cdots \otimes_{\mathbb{K}}  \mathsf{U}_{d}\rangle=0$,  then $\mathsf{W} \subseteq \sum_{i=1}^d \mathsf{U}_{i}^{\perp} \otimes \mathsf{V}_{i}$ where $\mathsf{U}_{i}^{\perp} \coloneqq ( (\mathbb{K}^{n_i})^\ast/\mathsf{U}_i)^\ast$.  Hence any $T\in \mathsf{W}$ admits a decomposition \eqref{lem4-8:eq1} if we choose $r_i = \dim \mathsf{U}_i^{\perp}$ and let $\{a_{i1},\dots,  a_{ir_i}\}$ be a basis of $\mathsf{U}_i^{\perp}$ for $1 \le i \le d$.  Thus we have $t = \sum_{i=1}^d \dim \mathsf{U}_i^\perp = \sum_{i=1}^d r_i \ge s$.

Conversely,  suppose that $r_i$'s and $a_{ij}$'s are given so that $s = \sum_{i=1}^d r_i$ and \eqref{lem4-8:eq1} holds for each $T \in \mathsf{W}$.  We define 
\[
\mathsf{U}_i \coloneqq \left( \spa  \{ a_{i1},\dots,  a_{ir_i} \} \right)^\perp  \coloneqq \left(\mathbb{K}^{n_i}/\spa  \{ a_{i1},\dots,  a_{ir_i} \} \right)^\ast \subseteq (\mathbb{K}^{n_i})^\ast,\quad 1 \le i \le d.
\]
Then clearly we have $\langle \mathsf{W}, \mathsf{U}_{1} \otimes_{\mathbb{K}}  \cdots \otimes_{\mathbb{K}}   \mathsf{U}_{d}\rangle = 0$ and this implies $s = \sum_{i=1}^d r_i \ge t$.
\end{proof}
\begin{lemma}\label{lem4-9}
For any field $\mathbb{K}$ and linear subspace $\mathsf{W} \subseteq \mathbb{K}^{n_1} \otimes_{\mathbb{K}}  \cdots \otimes_{\mathbb{K}}  \mathbb{K}^{n_d}$,  there exist $n_{d+1},n_{d+2} \in \mathbb{N}$ and $T_{\mathsf{W}}\in \mathbb{K}^{n_1} \otimes_{\mathbb{K}}   \cdots \otimes_{\mathbb{K}}  \mathbb{K}^{n_{d+2}}$ such that $\sr_{\mathbb{K}}(\mathsf{W}) = \sr_{\mathbb{K}}(T_{\mathsf{W}})$. 
\end{lemma}
\begin{proof} 
Assume $\dim \mathsf{W} = n$.  Let $T_1,\dots,  T_n$ be a basis of $\mathsf{W}$.  We denote $m \coloneqq \max \left\lbrace n_i: 1 \le i \le d \right\rbrace+1$.  Suppose that $\{e_{ij}: 1 \le i \le m,\; 1 \le j \le n \}$ is a basis of $\mathbb{K}^{mn}$.  We consider
\[
T_{\mathsf{W}} \coloneqq \sum_{j=1}^{n} T_{j}\otimes  \left( \sum_{i=1}^{m}  e_{ij} \otimes e_{ij} \right) \in \mathbb{K}^{n_1} \otimes_{\mathbb{K}}   \cdots \otimes_{\mathbb{K}}  \mathbb{K}^{n_d} \otimes_{\mathbb{K}}  \mathbb{K}^{mn} \otimes_{\mathbb{K}}  \mathbb{K}^{mn}. 
\]
By Lemma~\ref{lem4-8},  it is clear that $\sr_{\mathbb{K}}(T_{\mathsf{W}}) \le \sr_{\mathbb{K}}(\mathsf{W})$.  Thus, it suffices to prove $\sr_{\mathbb{K}}(T_{\mathsf{W}}) \ge \sr_{\mathbb{K}}(\mathsf{W})$. 

Let $T_{\mathsf{W}}=\sum_{i=1}^{d+2}\sum_{j=1}^{r_i} a_{ij}\otimes S_{ij}$ be a slice rank decomposition of $T_{\mathsf{W}}$,  where $a_{ij} \in \mathbb{K}^{n_i}$ and $S_{ij} \in \mathsf{V}_i$ for $1 \le i \le d+2$ and $1\le j \le r_i$.  Here $\mathsf{V}_i$ is defined as in \eqref{eq:Vi}.  We denote 
\[
\mathsf{W}' \coloneqq \sum_{i=1}^{d }\sum_{j=1}^{r_i}  \left\langle a_{ij}\otimes \mathsf{V}_i,  (\mathbb{K}^{mn})^\ast \otimes_{\mathbb{K}}  (\mathbb{K}^{mn})^\ast \right\rangle \subseteq \mathbb{K}^{n_1} \otimes_{\mathbb{K}}  \cdots \otimes_{\mathbb{K}}  \mathbb{K}^{n_d}.
\]
We observe that each element of $\mathsf{W}'$ is a linear combination of $a_{ij} \otimes R_{ij}$ for some 
\[
R_{ij} \in \langle \mathsf{V}_{i},  (\mathbb{K}^{mn})^\ast \otimes_{\mathbb{K}}  (\mathbb{K}^{mn})^\ast \rangle, \; 1\le i \le d,\; 1 \le j \le r_i.
\]
Without loss of generality,  we may assume $T_{1},\cdots, T_{l}\in \mathsf{W}'$ and $T_{l+1},\cdots,T_{n}\notin \mathsf{W}'$ for some $0 \le l \le n$. 
In the following,  we regard $a_{ij}$'s (resp.  $T_k$'s) as linear (resp.  multilinear) polynomials.  Since $T_{l+1}, \dots,  T_{n} \not\in \mathsf{W}'$,  we conclude that $T_{l+1}\cdots T_{n}  \notin J$, where $J$ is the prime ideal generated by  linear forms $a_{ij},  1\le i \le d,  1 \le j \le r_i$.  By Hilbert Nullstellensatz,  there exist $v_i\in \overline{\mathbb{K}}^{n_i}$,  $1\le i\le d$, such that 
\[
a_{ij}(v_{i})=0,\; T_{s}(v_{1},\cdots,v_{d})\not=0,\;1\le j\le r_{i},\; l+1\le s\le d.
\]
Therefore,  we obtain 
\begin{align}
T_{\mathsf{W}}(v_{1},\cdots,v_{d})&=\sum_{i=1}^{d+2}\sum_{j=1}^{r_{i}} (a_{ij} \otimes S_{ij})( v_1,\dots, v_{d} ) \nonumber \\
&= \sum_{j=1}^{r_{d+1}}a_{d+1,j} \otimes S_{d+1,j}(v_1,\dots, v_{d})
+ 
\sum_{j=1}^{r_{d+2}}a_{d+2,j} \otimes S_{d+2,j}(v_1,\dots, v_{d})
\label{lem4-9:eq1}
\end{align}
On the other hand,  by the construction of $T_{\mathsf{W}}$,  we also have 
\begin{equation}\label{lem4-9:eq2}
T_{\mathsf{W}}(v_{1},\cdots,v_{d})=
\sum_{s=1}^{n} T_{s} (v_1,\dots,  v_d)  \left( \sum_{i=1}^{m}  e_{is}\otimes e_{is} \right) = \sum_{s=l+1}^{n} T_{s} (v_1,\dots,  v_d)  \left( \sum_{i=1}^{m}  e_{is}\otimes e_{is} \right).
\end{equation}
Both \eqref{lem4-9:eq1} and \eqref{lem4-9:eq2} are decompositions of the matrix $T_{\mathsf{W}}(v_{1},\cdots,v_{d}) \in \overline{\mathbb{K}}^{mn \times mn}$ into the sum of rank one matrices.  Moreover,  \eqref{lem4-9:eq2} is clearly a rank decomposition. This implies that $(n-l)m \le r_{d+1} + r_{d+2}$ which implies $l=n$ since $r_{d+1}+r_{d+2}\le\sr_{\mathbb{K}}(T_{\mathsf{W}}) <  m$ by construction. Hence all $T_{i}$ are in $\mathsf{W}'$,  and this implies $\operatorname{SR}_{\mathbb{K}} (\mathsf{W}) \le \operatorname{SR}_{\mathbb{K}} (\mathsf{W}') \le \operatorname{SR}_{\mathbb{K}} ( T_\mathsf{W})$.
\end{proof}
A direct application of Lemmas~\ref{lem4-9} and \ref{lem4-3} implies the additivity of the slice rank of linear subspaces with respect to the direct sum.
\begin{proposition}[Additivity of slice rank of subspaces] \label{coro4-10}
For any field $\mathbb{K}$ and linear subspaces $\mathsf{W}_1 \subseteq \mathbb{K}^{n_1} \otimes_{\mathbb{K}}  \cdots \otimes_{\mathbb{K}}  \mathbb{K}^{n_d}$,  $\mathsf{W}_{2}\subseteq \mathbb{K}^{m_1} \otimes_{\mathbb{K}}  \cdots \otimes_{\mathbb{K}}   \mathbb{K}^{m_d}$,  we have 
\[
\sr_{\mathbb{K}}(\mathsf{W}_{1}\oplus \mathsf{W}_{2})= \sr_{\mathbb{K}}(\mathsf{W}_1)+\sr_{\mathbb{K}}(\mathsf{W}_2).
\]
\end{proposition}
Conjecture~\ref{conj1-5} can be easily confirmed by Lemma~\ref{lem4-9}.
\begin{theorem}[Stability of slice rank of subspaces I]
\label{thm:slrank-conj}
Let $\mathbb{K}$ be a field.  For any positive integer $d$ and linear subspace $\mathsf{W} \subseteq \mathbb{K}^{n_1} \otimes_{\mathbb{K}} \cdots \otimes_{\mathbb{K}} \mathbb{K}^{n_d}$,  we have 
\[
\sr_{\mathbb{K}}(\mathsf{W}) \le \left( \frac{d}{2} + 1\right)\sr_{\overline{\mathbb{K}}}\left( \mathsf{W} \otimes \overline{\mathbb{K}} \right).
\]
\end{theorem}
\begin{proof}
According to the proof of Lemma~\ref{lem4-9},  there is a tensor $T_{\mathsf{W}} \in \mathbb{K}^{n_1} \otimes_{\mathbb{K}} \cdots \otimes_{\mathbb{K}}  \mathbb{K}^{n_{d+2}}$ such that $\sr_{\mathbb{K}}(\mathsf{W})= \sr_{\mathbb{K}}(T_{\mathsf{W}})$ and $\sr_{\overline{\mathbb{K}}}(\mathsf{W} \otimes \overline{\mathbb{K}}) = \sr_{\overline{\mathbb{K}}}\left( T_{\mathsf{W}}^{\overline{\mathbb{K}}}\right)$.  Now the desired inequality follows immediately from \eqref{rem4-5:eq1}.
\end{proof}
We remark that Theorem~\ref{thm:slrank-conj} is proved for $d = 2$ in \cite{adiprasito2021schmidt} by a different approach.  To conclude this section,  we define for each $\mathsf{W} \subseteq \mathbb{K}^{n_1} \otimes_{\mathbb{K}}  \cdots \otimes_{\mathbb{K}} \mathbb{K}^{n_d}$ and positive integer $1 \le k \le \dim \mathsf{W}$ the following quantity:
\begin{equation}\label{eq:SRkW}
\sr_{k,\mathbb{K}}( \mathsf{W} ) \coloneqq \max \{ \sr_{\mathbb{K}} (\mathsf{V}): \mathsf{V} \subseteq \mathsf{W},\; \dim \mathsf{V} = k \}.
\end{equation}
We will prove below that for $d = 2$,  $\sr_{k,\mathbb{K}}(\mathsf{W})$ lies between constant multiples of $\sr_{\mathbb{K}}(\mathsf{W})$ and $\sr_{\overline{\mathbb{K}}}(\mathsf{W}\otimes\overline{\mathbb{K}})$.
\begin{lemma}\label{lem4-12}
For any field $\mathbb{K}$ and linear subspace $\mathsf{W} \subseteq \mathbb{K}^{n_1 \times n_2}$ and integer $1 \le k \le \dim \mathsf{W}$,  we have 
\[
\sr_{\mathbb{K}}(W)\le 2 \sr_{k,\mathbb{K}}(\mathsf{W}).
\]
\end{lemma}
\begin{proof}
The argument is similar to that in the proof of \cite[Theorem~1.17]{adiprasito2021schmidt},  thus we omit it.
\end{proof}
\begin{proposition}[Stability of slice rank of subspaces II]\label{thm4-11}
Let $\mathbb{K}$ be a field.  For any linear subspace $\mathsf{W} \subseteq \mathbb{K}^{n_1 \times n_2}$ and integer $1 \le k \le \dim \mathsf{W}$,  we have 
\[
\frac{1}{2}\sr_{\mathbb{K}}(\mathsf{W}) \le \sr_{k,\mathbb{K}}(\mathsf{W})\le 2\sr_{k,\overline{\mathbb{K}}}(\mathsf{W} \otimes \overline{\mathbb{K}})\le 2 \sr_{\overline{\mathbb{K}}}(\mathsf{W}\otimes\overline{\mathbb{K}}).
\]
\end{proposition}
\begin{proof}
The first inequality is obtained by Lemma~\ref{lem4-12} and the last is clear by definition.  For the second inequality,  let $\mathsf{V}$ be a $k$-dimensional subspace of $\mathsf{W}$ such that $\sr_{k,\mathbb{K}}(\mathsf{W}) = \sr_{\mathbb{K}}(\mathsf{V})$.  Then Theorem~\ref{thm:slrank-conj} implies
\[
\sr_{k,\mathbb{K}}(\mathsf{W}) = \sr_{\mathbb{K}}(\mathsf{V}) \le 2\sr_{\overline{\mathbb{K}}}(\mathsf{V} \otimes \overline{\mathbb{K}}) \le 2 \sr_{k,\overline{\mathbb{K}}}(\mathsf{W} \otimes \overline{\mathbb{K}}). 
\]
\end{proof}
It is natural to expect that for any integer $d \ge 3$,  there is also a constant $C \coloneqq C(d) > 0$ such that  
\begin{equation}\label{eq:d3}
\sr_{\mathbb{K}}(\mathsf{W}) \le C \sr_{k,\mathbb{K}}(\mathsf{W}).
\end{equation}
If \eqref{eq:d3} holds,  then our proof of Proposition~\ref{thm4-11} can be generalized for $d \ge 3$.  Unfortunately,  this is not always the case by the proposition that follows. 
\begin{proposition}
For $d \ge 3$ and $k =1$,  there is no constant $C \coloneqq C(d) > 0$ such that \eqref{eq:d3} holds.  
\end{proposition}
\begin{proof}
We only prove the non-existence of $C$ for $d = 3$,  as the proof for larger $d$ is the same. Let $n$ be a positive integer.  We consider tensors 
\[
T_s \coloneqq I_{n} \otimes e_k \in \mathbb{K}^n \otimes_{\mathbb{K}}  \mathbb{K}^n \otimes_{\mathbb{K}}  \mathbb{K}^n,\quad 1 \le s \le n,
\]
where $I_n$ is the $n \times n$ identity matrix and $\{e_1,\dots,  e_n\}$ is a basis of $\mathbb{K}^n$.  Denote $\mathsf{W} \coloneqq \spa \{T_1,\dots,  T_n\}$.  It is clear that $\sr_{1,\mathbb{K}}(\mathsf{W}) = 1$.

Let $1 \le r_1,r_2,r_3\le n$ be positive integers such that $r_1 + r_2 + r_3 = \sr_{\mathbb{K}} (\mathsf{W})$ and 
\[
T_s \in \mathsf{U}_1 \otimes_{\mathbb{K}} \mathbb{K}^n \otimes_{\mathbb{K}}  \mathbb{K}^n + \mathbb{K}^{n} \otimes_{\mathbb{K}} \mathsf{U}_2  \otimes_{\mathbb{K}} \mathbb{K}^n + \mathbb{K}^{n} \otimes_{\mathbb{K}} \mathbb{K}^n \otimes_{\mathbb{K}}  \mathsf{U}_3,\quad 1 \le s \le n,
\] 
for some $\mathsf{U}_i \subseteq \mathbb{K}^n$ of dimension $r_i$,  $i = 1,2,3$.  If $r_3 < n$,  then there exists $f\in \mathsf{U}_3^\perp \coloneqq (\mathbb{K}^n/\mathsf{U}_1)^\ast$ such that $f(e_{k_0}) = 1$.  This implies 
\[
I_n  = \langle T_{k_0},  f \rangle  \in \mathsf{U}_1 \otimes_{\mathbb{K}} \mathbb{K}^n + \mathbb{K}^n \otimes_{\mathbb{K}} \mathsf{U}_2.
\]
Hence we obtain $\sr_{\mathbb{K}}(\mathsf{W})\ge r_1 + r_2  \ge \R_{\mathbb{K}}(I_n) = n$.  If $r_3 = n$,  then we also have $\sr_{\mathbb{K}}(\mathsf{W})\ge n$.  Since $n$ can be arbitrarily large and $\sr_{1,\mathbb{K}}(\mathsf{W}) = 1$,  there is no constant $C > 0$ such that $\sr_{\mathbb{K}}(\mathsf{W}) \le C \sr_{1,\mathbb{K}}(\mathsf{W})$.
\end{proof}
We notice that \eqref{eq:d3} holds for any $d$ when $k = \dim \mathsf{W}$.  Thus,  it is interesting to study the critical value $k_0$ such that \eqref{eq:d3} holds for $k \ge k_0$.  According to Theorem~\ref{thm:slrank-conj} and Proposition~\ref{thm4-11},  it is also reasonable to expect that $\sr_{k,\mathbb{K}}(\mathsf{W})$ is stable under field extensions as well.  
\begin{conjecture}[Stability of slice rank of subspaces III]\label{conj6}
There is a constant $C \coloneqq C (d)>0$ such that for any linear subspace $\mathsf{W} \subseteq \mathbb{K}^{n_1} \otimes_{\mathbb{K}} \cdots \otimes_{\mathbb{K}} \mathbb{K}^{n_d}$ and positive integer $k \le \mathsf{W}$,  we have 
\[
\sr_{k,\overline{\mathbb{K}}}(\mathsf{W} \otimes_{\mathbb{K}} \overline{\mathbb{K}})\le C\sr_{k,\mathbb{K}}(\mathsf{W}).
\]
\end{conjecture}

\appendix 
\section*{Appendix}
\section{Algebraic function field}
We provide a brief introduction to the theory of algebraic function fields.  Interested readers are referred to \cite{stichtenoth2009algebraic} for more details.  An \emph{algebraic function field of one variable} over $\mathbb{K}$ is a field extension $\mathbb{F}/\mathbb{K}$ such that there is a transcendental element $f\in\mathbb{F}$ and $[\mathbb{F}:\mathbb{K}(f)] < \infty$.  We recall that any valuation ring $(\mathcal{O},\mathfrak{m},  v_{\mathfrak{m}})$ such that $\mathbb{K} \subsetneq\mathcal{O}\subsetneq\mathbb{F}$ and $F(\mathcal{O})=\mathbb{F}$ is a \emph{discrete valuation ring (DVR)}.  Here $\mathfrak{m}$ is the maximal ideal of $\mathcal{O}$,  $F(\mathcal{O})$ is the field of fractions of $\mathcal{O}$ and $v_{\mathfrak{m}}$ is the valuation of $F(\mathcal{O})$.  The maximal ideal $\mathfrak{m}$ of $\mathcal{O}$ is called a \emph{place} of $\mathbb{F}/\mathbb{K}$.  The \emph{degree of $\mathfrak{m}$} is defined as $\deg(\mathfrak{m}) \coloneqq \dim_{\mathbb{K}}(\mathcal{O}/ \mathfrak{m} )$.  

Let $\Div(\mathbb{F}/\mathbb{K})$ be the free abelian group generated by places of $\mathbb{F} /\mathbb{K}$.  An element in $\Div(\mathbb{F}/\mathbb{K})$ is called a \emph{divisor} and $\Div(\mathbb{F}/\mathbb{K})$ is called the \emph{divisor group}.  For each $D \in \sum_{\mathfrak{m}} n_{\mathfrak{m}} \mathfrak{m}$ where $n_{\mathfrak{m}} = 0$ for all but finitely many $\mathfrak{m}$'s,  we define the \emph{degree} of $D$ by $\deg(D) \coloneqq \sum_{\mathfrak{m}} n_{\mathfrak{m}} \deg(\mathfrak{m})$.  If $n_{\mathfrak{m}} \ge 0$ for each place $\mathfrak{m}$,  we denote $D \ge 0$.  Given an element $f \in \mathbb{F}$,  the \emph{principal divisor} generated by $f$ is $(f) \coloneqq \sum_{\mathfrak{m}} v_{\mathfrak{m}}(f)\mathfrak{m}$.

For each divisor $D$,  the \emph{Riemann-Roch space} is defined as 
\[
\mathscr{L}(D)  \coloneqq \left\{f \in\mathbb{F}: (f)+ D\ge 0\right\}\cup\left\{0\right\}
\]
We notice that $\mathscr{L}(D)$ is a finite dimensional $\mathbb{K}$-subspace of $\mathbb{F}$ and denote $\ell(D) \coloneqq \dim_{\mathbb{K}} \mathscr{L}(D)$.  The  \emph{genus} of $\mathbb{F} /\mathbb{K}$ is $g \coloneqq \max_{D \in \Div(\mathbb{F}/\mathbb{K})}  \left\lbrace \deg(D)- \ell(D)+1 \right\rbrace$.  For ease of reference,  we record below some basic lemmas that will be needed in the paper.
\begin{lemma}\cite[Theorem~1.5.17]{stichtenoth2009algebraic}\label{lem2-3}
Let $g$ be the genus of $\mathbb{F}/\mathbb{K}$.  For each $D \in  \Div(\mathbb{F}/\mathbb{K})$ with $\deg(D) \ge 2g - 1$,  we have $\ell(D)= \deg(D)+1-g$.
\end{lemma}
\begin{lemma}\cite[Theorem~1.4.11]{stichtenoth2009algebraic} \label{lem2-4}
    All principal divisors have degree $0$.
\end{lemma}
\begin{lemma}[Weak approximation]\cite[Theorem~1.3.1]{stichtenoth2009algebraic}\label{lem2-5}
Let $\mathbb{F} /\mathbb{K}$ be a function field.  Given distinct places $\mathfrak{m}_1,\dots, \mathfrak{m}_n$ of $\mathbb{F}/\mathbb{K}$,  $f_{1},\cdots,f_{n}\in \mathbb{F}$ and $r_{1},\dots,r_{n}\in\mathbb{Z}$,  there exists some $f \in \mathbb{F}$ such that $v_{\mathfrak{m}_i} (f - f_i) = r_i$ for each $1 \le i \le n$. 
\end{lemma}
\begin{lemma}\cite[Corollary~1.4.12]{stichtenoth2009algebraic}\label{lem2-6}
For any $D\in \Div(\mathbb{F}/\mathbb{K})$ with $\deg(D)<0$,  we have $\ell(D)=0$.
\end{lemma}
A divisor $D\in \Div(\mathbb{F}/\mathbb{K})$ is \emph{non-special} if $\ell(D) = \deg(D)+1-g$, otherwise it is \emph{special}.
\begin{lemma}\cite[Lemma~2.1]{ballet1999curves}
Let $q$ be a prime power. Let $\mathbb{F}/\mathbb{F}_{q}$ be an algebraic function field of genus $g$ containing at least $g+1$ places of degree one.  Denote by $P_{1}(F/\mathbb{F}_{q})$ the set of places of degree one.  For any $S\subseteq P_{1}(\mathbb{F}/\mathbb{F}_{q})$ with $\vert S\vert\ge g+1$,  there exists a non-special divisor $D$ such that $\deg (D) = g-1$ and $supp (D) \subseteq S$. Here $supp(D)$ is the set of place with non-zero coefficient in $D$.
    \label{lem2-7}
\end{lemma}
\begin{lemma}\cite[Theorem]{garcia1995tower}\label{lem2-10}
For each positive integer $k$,  there exists an algebraic function field $\mathbb{F}/\mathbb{F}_{q^{2}}$ whose genus is 
\begin{equation}\label{equ3}
g_k \coloneqq \begin{cases}
q^{k}+q^{k-1}-q^{(k+1)/2}-2q^{(k-1)/2}+1   & \text{if $k$ is odd} \\
q^{k}+q^{k-1}-\frac{1}{2}q^{k/2 + 1}-\frac{3}{2}q^{k/2}-q^{k/2-1}+1   & \text{otherwise} 
\end{cases}.
\end{equation}
Moreover,  if $k \ge 3$,  then $\mathbb{F}$ admits at least 
\begin{equation}\label{eq:Nk}
N_k \coloneqq (q^{2}-1)q^{k-1}+2q\ge(q-1)g_k 
\end{equation}
degree one places. 
\end{lemma}
\begin{lemma}\cite[Corollary~5.2.10]{stichtenoth2009algebraic} \label{lem2-11}
Any algebraic function field $\mathbb{F}/\mathbb{F}_{q}$ of genus $g$ admits at least one place of degree $r$ for each $r\ge 4g+3$.
\end{lemma}
\section{Number of rational points}
\begin{theorem}\cite[Theorem~11.1]{gho2008etale}\label{thm3-6}
Let $X$ be an $n$-dimensional irreducible projective subvariety of $\mathbb{P}_{ \overline{\mathbb{F}}_q }^{N}$ defined over $\mathbb{F}_q$.  Denote $d\coloneqq \deg(X)$.  There is a constant number $C \coloneqq C(X) >0$ such that
\[
\left\lvert  | X(\mathbb{F}_q)|- \frac{q^{n+1} - 1}{q-1}  \right\rvert \le (d-1)(d-2)q^{n- \frac{1}{2}}+Cq^{n-1}.
\]
\end{theorem}

\section*{Acknowledgments} 
We would like to thank Daniel Zhu and Guy Moshkovitz for their insightful comments,  which helped us to fill gaps in the proofs of Lemma~\ref{Poly-finite} and Theorem~\ref{thm2-13}.  We are grateful to Jan Draisma for his valuable suggestions to improve the proof of Lemma~\ref{lemPR=1},  and to Thomas Karam for his detailed feedback on the draft.  We also appreciate anonymous reviewers for their helpful suggestions that improved the quality of the paper.

\bibliographystyle{amsplain}


\begin{dajauthors}
\begin{authorinfo}[QYC]
  Qiyuan Chen\\
State Key Laboratory of Mathematical Sciences \\
Academy of Mathematics and Systems Science \\
  Chinese Academy of Sciences \\
  Beijing 100190,  China\\
  chenqiyuan\imageat{}amss\imagedot{}ac\imagedot{}cn \\
  \url{https://qiyuannickchen.github.io/QYChen.github.io//}
\end{authorinfo}
\begin{authorinfo}[KY]
  Ke Ye\\
  Associate Professor\\
State Key Laboratory of Mathematical Sciences \\
Academy of Mathematics and Systems Science \\
  Chinese Academy of Sciences \\
  Beijing 100190,  China\\
  keyk\imageat{}amss\imagedot{}ac\imagedot{}cn \\
  \url{https://sites.google.com/site/keyeshomepage/home}
\end{authorinfo}

\end{dajauthors}

\end{document}